\documentclass[12pt]{amsart}

\usepackage[inner=25mm, outer=25mm, textheight=225mm]{geometry}

 \usepackage[dvips]{graphicx}

\usepackage{amsmath,graphics,color}
\usepackage{amsfonts,amssymb}
\usepackage{epstopdf}
\usepackage{xypic,hyperref}

\theoremstyle{plain}
\newtheorem*{theorem*}{Theorem}
\newtheorem*{lemma*} {Lemma}
\newtheorem*{corollary*} {Corollary}
\newtheorem*{proposition*} {Proposition}
\newtheorem{theorem}{Theorem}[section]
\newtheorem{lemma}[theorem]{Lemma}
\newtheorem{corollary}[theorem]{Corollary}
\newtheorem{proposition}[theorem]{Proposition}
\newtheorem{conjecture}[theorem]{Conjecture}
\newtheorem{question}[theorem]{Question}

\theoremstyle{remark}
\newtheorem*{remark}{Remark}

\newtheorem*{example}{Example}
\newtheorem*{claim}{Claim}

\theoremstyle{definition}

\def\be{\begin{equation}}
\def\ee{\end{equation}}
\def\co{\colon}
\def\RR{\mathcal{R}}

\def \K {\mathbf{K}}\def \Z {\mathbf{Z}}
\def \G {\mathbf{G}}

\def\R{\mathbb{R}}\def\genus{\op{genus}}
\def\eps{\epsilon}

\def\gl{\op{GL}}
\def\Q{\mathbb{Q}}\def\K{\mathbb{K}}
\def\id{\op{id}}
\def\Z{\mathbb{Z}}

\def\N{\mathbb{N}}
\def\part{\partial}\def\ll{\langle}\def\rr{\rangle}

\def\a{\alpha}

\def\bp{\begin{pmatrix}}
\def\sm{\setminus}\def\ep{\end{pmatrix}}
\def\bn{\begin{enumerate}}

\def\rank{\op{rank}}
\def\en{\end{enumerate}}\def\ba{\begin{array}}
\def\ea{\end{array}}

\def\S{\Sigma}
\def\a{\alpha}

\def\hull{\op{Hull}}

\def\fr12{\frac{1}{2}}

\def\ker{\op{Ker}}

\def\hom{\op{Hom}}

\def\kt{\K\tpm}
\def\deg{\op{deg}}

\def\tpm{[t^{\pm 1}]}

\def\zt{\Z\tpm}\def\qt{\Q\tpm}

\def\G{\Gamma}
\def\ol{\overline}

\def\op{\operatorname}

\def\zt{\Z[t^{\pm 1}]}

\def\K{\mathbb{K}}

\def\th{\op{th}}

\def\cmtbf#1{} \def\cmt#1{}

\def\NN{\mathcal{N}}
\def\PP{\mathcal{P}}

\def\TT{\mathcal{T}}
\def\MM{\mathcal{M}}
\def\QQ{\mathcal{Q}}
\def\XX{\mathcal{X}}
\def\YY{\mathcal{Y}}
\def\VV{\mathcal{V}}
\def\GG{\mathcal{G}}
\def\hull{\op{conv}}
\def\wti{\widetilde}
\def\what{\widehat}
\def\sym{{\op{sym}}}
\def\mfp{\mathfrak{P}}
\def\mfg{\mathfrak{G}}
\def\TT{\mathcal{T}}
\def\th{\op{th}}

\begin{document}

\title[Groups and marked polytopes]{Two-generator one-relator groups\\ and marked polytopes}

\author{Stefan Friedl}
\address{Fakult\"at f\"ur Mathematik\\ Universit\"at Regensburg\\   Germany}
\email{sfriedl@gmail.com}

\author{Stephan Tillmann}
\address{School of Mathematics and Statistics\\ The University of Sydney\\ NSW 2006 Australia} 
\email{stephan.tillmann@sydney.edu.au} 

\date{\today}
\def\subjclassname{\textup{2000} Mathematics Subject Classification}
\expandafter\let\csname subjclassname@1991\endcsname=\subjclassname \expandafter\let\csname
subjclassname@2000\endcsname=\subjclassname 

\date{\today}
\begin{abstract} 
We  use Fox calculus to assign a marked polytope to a `nice' group presentation with two generators and one relator. Relating the marked vertices to Novikov--Sikorav homology we show that they  determine the Bieri--Neumann--Strebel invariant of the group. Furthermore we  show that in many cases the marked  polytope is an invariant of the underlying group and that in those cases the marked polytope also determines the minimal complexity of all the associated HNN-splittings. 
\end{abstract}
\maketitle

\[ \mbox{\emph{Dedicated to the memory of Tim Cochran}}\]
\vspace{0.5cm}

\section{Summary of results}
In this paper, a  $(2,1)$--presentation is a group presentation $\pi=\ll x,y\,|\, r\rr$ 
with  two generators and one relator. A $(2,1)$--presentation $\pi$ naturally gives rise to a group, which we denote $G_\pi$.
We say that a $(2,1)$--presentation $\pi$ is \emph{nice} if it satisfies the following conditions:
\bn
\item $r$ is a non-empty,  cyclically reduced word, and
\item $b_1(G_\pi)=2$.
\en
To a nice $(2,1)$--presentation $\pi=\ll x,y\,|\, r\rr$ we will associate
a marked polytope $\MM_\pi$ in $H_1(G_\pi;\R)$.
A marked polytope  is a polytope  together with a (possibly empty) set of marked vertices.
Now we give an informal outline of the definition of $\MM_\pi$ (see also Figure~\ref{fig:brown-intro}), a formal definition is given in Section~\ref{section:defpolytopepi}.

Identify $H_1(G_\pi;\Z)$ with $\Z^2$ such that $x$ corresponds to $(1,0)$ and $y$ corresponds to $(0,1)$. 
Then the relator $r$ determines a discrete walk on the integer lattice in $H_1(G_\pi;\R),$ and the marked polytope $\MM_\pi$ is obtained from the convex hull of the trace of this walk:
\bn
\item Start at the origin and walk across $\Z^2$ reading the word $r$ from the left.
\item Take the convex hull $\mathcal{C}$ of the set of all lattice points reached by the walk.
\item Mark precisely those vertices of $\mathcal{C}$ that the walk passes through exactly once.
\item Consider the unit squares that are completely contained in $\mathcal{C}$ and which touch a vertex of $\mathcal{C}$. The set of vertices of $\MM_\pi$ is defined as the set of midpoints of all of these squares, and a vertex of $\MM_\pi$ is marked precisely when all the corresponding vertices of $\mathcal{C}$ are marked.
\en
Figure~\ref{fig:brown-intro} illustrates the construction of the marked polytope for the presentation
$$\pi=\ll x,y\,|\, yx^4yx^{-1}y^{-1}x^2y^{-1}x^{-2}y^2xy^{-1}xy^{-1}x^{-1}y^{-2}x^{-3}y^2x^{-1}\rr.$$

\begin{figure}[h]
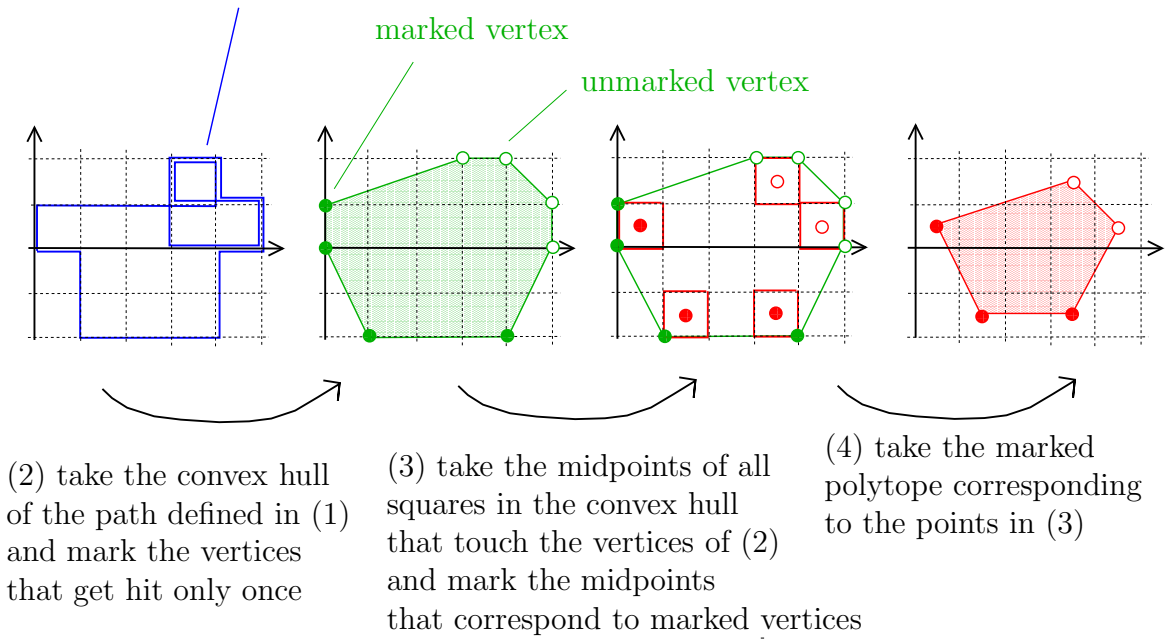\label{fig:brown-intro}
\caption{Marked polytope $\MM_\pi$ for a presentation $\pi$}
\end{figure}

We expect that the marked polytope $\MM_\pi$ contains interesting information about the group $G_\pi$; an example of this is given by our first main result.
The Bieri--Neumann--Strebel invariant $\S(G)$ of the
 finitely generated group $G$ is an open subset of the `sphere' $S(G):=(\hom(G,\R)\sm \{0\})/\R_{>0}$. (See Section~\ref{section:bns} for more details.) It turns out that $\MM_\pi$ determines the Bieri--Neumann--Strebel invariant of $G_\pi$.
In order to state this result we need one more definition. Given the  polytope $\MM$ in the vector space $V,$ we say that the homomorphism $\phi\in \hom(V,\R)$ \emph{pairs maximally with the vertex $v$} if $\phi(v)>\phi(w)$ for all vertices $w\ne v$.

\begin{theorem}\label{mainthm}
Let $\pi=\ll x,y\,|\, r\rr$ be a nice $(2,1)$--presentation.
A non-trivial class $\phi\in H^1(G_\pi;\R)$  represents an element in $\Sigma(G_\pi)$ if and only if $\phi$ pairs maximally with a marked vertex of $\MM_\pi$. 
\end{theorem} 

The well-versed reader might be excused for a sense of d\'eja vu: the
theorem can be viewed as a reformulation of Brown's algorithm \cite[Theorem~4.3]{Brn87} and it is closely related to~\cite[Theorem~7.3]{BR88}.
The key observation in our proof is a reformulation of $\MM_\pi$ in terms of the Fox derivatives $r_x=\frac{\partial r}{\partial x}$ and $r_y=\frac{\partial r}{\partial y}$, leading to a straightforward proof of Theorem~\ref{mainthm} using the generalised Novikov rings of Sikorav~\cite{Si87}.

The marked polytope $\MM_\pi$ associated to the $(2,1)$--presentation $\pi=\ll x,y\,|\, r\rr$ depends a priori on the presentation $\pi$ and not just on the isomorphism type of the group $G_\pi$. Lemma~\ref{lem:polytopetranslate} shows that if we replace $r$ by a cyclic permutation of $r$, then the resulting  marked polytope in $H_1(G_\pi;\R)$ is a translate of the original marked polytope. 
We suspect that this is the only indeterminacy. More precisely, we propose the following conjecture.

\begin{conjecture}\label{mainconj}
If $G$ is a group admitting a nice $(2,1)$--presentation $\pi$,
then up to translation the marked polytope $\MM_\pi\subset H_1(G;\R)$ is an invariant of $G$.
\end{conjecture}

The difficulty in proving the conjecture is that to the best of our knowledge there is no good theory which relates two 
$(2,1)$--presentations of a group. For example, Zieschang  \cite[p.~36]{Zi70}
and also MacCool--Pietrowski \cite{MP73} showed 
that there exist $(2,1)$--presentations   $ \ll x,y\,|\, r\rr$ and $\ll x',y'\,|\, r'\rr$ representing isomorphic groups,
but such that no isomorphism is induced by an isomorphism of the free groups $\ll x,y\rr$ and $\ll x',y'\rr$. 

Let $\GG$ denote the class of all groups that are torsion-free and elementary amenable. Then $\GG$ contains in particular all torsion-free solvable groups. A group $G$ is \emph{residually $\GG$} if given any non-trivial element $g\in G,$ there exists a homomorphism $\a\co G\to \G$ with $\G\in \GG$ such that $\a(g)$ is non-trivial.

The following can be seen as evidence towards a positive answer to Conjecture~\ref{mainconj}.

\begin{theorem}\label{mainthm2}\label{thm:polytope}
Let  $G$ be a group admitting a nice $(2,1)$--presentation $\pi$.
If $G$ is residually $\GG$, then the polytope $\MM_\pi\subset H_1(G;\R)$ is an invariant of the group $G$ $($up to translation$)$.
\end{theorem}

We show in Lemma~\ref{lem:zbyfree} that a group $G$ satisfies the hypothesis of the theorem if there exists $[\phi]\in S(G)$ such that both $[\phi]$ and $[-\phi]$ lie in $\S(G)$. This is not as rare an occurrence as it might sound: Dunfield and D. Thurston \cite[Section~6]{DT06} give strong evidence for the conjecture that `most' groups with a nice $(2,1)$--presentation have this property. Moreover, if $G$ is the fundamental group of an aspherical 3--manifold, then it is conjectured that $G$ is in fact residually $\GG$ (see \cite{AFW13}).

As is perhaps to be expected, considering the authors' background,
the motivation for introducing and studying the marked polytope $\MM_\pi$ comes from 3--manifold topology. In \cite{FT15} we show that for many, possibly all, 3--manifolds such that the fundamental group admits a nice $(2,1)$--presentation, the polytope $\MM_\pi$ is dual to the unit ball of the Thurston norm \cite{Th86} of the 3--manifold, with the marked vertices dual to the fibered cones. However, the focus of this paper is on the following group theoretic analogue of the Thurston norm, which is of independent interest.

Given a group $G$ and an epimorphism $\phi\co G\to \Z,$ define $c(G,\phi)$ as the minimal rank of a group along which we can split 
$(G,\phi)$. We refer to Section~\ref{section:splittings} for details. We relate this quantity to the geometry of the polytope $\MM_\pi$ via the notion of \emph{thickness}.
Given the polytope $\PP$ in the vector space $V,$ the \emph{thickness of $\PP$ with respect to the homomorphism $\phi\co V\to \R$} is
\[ \th(\PP,\phi):=\max\{ \phi(p)-\phi(q)\,|\,p,q\in \PP\}.\]
In our setting, $\PP= \MM_\pi,$ $V = H_1(G_\pi;\R)\cong \R^2$ and $\phi\co G_\pi\to \Z$ induces a homomorphism $V\to \R$ denoted by the same letter. 

\begin{theorem}\label{mainthm3}
Let  $G$ be a group, which is residually $\GG$ and has the nice $(2,1)$--presentation $\pi.$
Then for every epimorphism $\phi\co G\to \Z$ we have
\[ c(G,\phi)=\th(\MM_\pi,\phi)+1.\]
\end{theorem}

It is straightforward to see that measuring thickness of a polytope gives rise to a seminorm. Thus we obtain the following corollary.

\begin{corollary}
Let  $G$ be a group, which is residually $\GG$ and has a nice $(2,1)$--presentation. Then 
\[ \ba{rcl} \hom(G;\Z)&\to &\Z_{\geq 0} \\
\phi&\mapsto & c(G,\phi)-1\ea \]
is a seminorm.
\end{corollary}

The results summarized thus far allow us to conclude the introduction with the following corollary, which has a conceptually simple proof.

\begin{corollary}
Let  $G$ be a group that admits a nice $(2,1)$--presentation $\pi$. 
Then either $G$ is  isomorphic to $\Z^2$ or there exists $\phi\in H^1(G;\Z),$ which does not lie in $\Sigma(G)$.
\end{corollary}

\begin{proof}
If there exists no $\psi\in H^1(G;\Z)$ with the property that both $\psi$ and $-\psi$ lie in $\Sigma(G)$, then we are clearly done. Now suppose such $\psi$ exists. It follows from   Lemma~\ref{lem:zbyfree} that $G$ is residually $\GG$, in particular $G$ is torsion-free.

If $c(G,\psi)=0$, then $G$ is a free group and $\Sigma(G)$ is well-known to be the empty set. If $c(G,\psi)=1$, then it follows from the definition of $c(G,\psi)$ that $\ker(\psi)$ is a group of rank one. Since $G$ is torsion-free it follows that
$\ker(\psi)\cong \Z$. Put differently,  $G$ is a semidirect product of $\Z$ with $\Z$. Since $b_1(G)=2$ we see that $G\cong \Z^2$. If $c(G,\psi)>1$,
then it follows from Lemma~\ref{lem:zbyfree} and Theorem~\ref{mainthm3} that 
 $\th(\MM_\pi,\psi)>0$. This implies that $\MM_\pi$ does not consist of a single point. Since $b_1(G)=2$ and since $\MM_\pi$ has vertices which lie in $H_1(G;\Z)/\mbox{torsion}\subset H_1(G;\R)$ it  follows that there exists a $\phi\in H^1(G;\Z)$ which does not pair maximally with a vertex of $\MM_\pi$. By Theorem~\ref{mainthm} this $\phi$ does not lie in $\Sigma(G)$.
\end{proof}

The paper is organized as follows.
In Section~\ref{section:twoonepres}, we prove some basic facts about marked polytopes and define the marked polytope associated to a nice $(2,1)$--presentation. In Section~\ref{sec:marked polytope via fox calculus}, it is shown how the marked polytope is related to the Fox derivatives of the relators. The proof of Theorem~\ref{mainthm} is given in Section~\ref{sec:proofs of first theorem}, an example in Section~\ref{section:example}, and the proof of Theorem~\ref{mainthm2} in Section~\ref{sec:proof of second theorem}. In Section~\ref{sec:proof of third theorem}, we relate thickness of polytopes to complexity of splittings and prove Theorem~\ref{mainthm3}. We discuss the case of groups which admit a $(2,1)$--presentation but for which the abelianization is not equal to $\Z^2$ in Section~\ref{section:b1}. Our paper is concluded with a list of open questions in Section~\ref{section:questions}. 

\subsection*{Convention.}
Given a ring $R$ we mean by a module a left $R$-module, unless stated otherwise.
Furthermore, we view elements in $R^n$ as row-vectors. An $k\times l$-matrix over $R$ induces a left $R$-module homomorphism $R^k\to R^l$ by right-multiplication on row-vectors.

\subsection*{Acknowledgment.} 
Most of the research was carried out while the first author was visiting the University of Sydney. The first author is very grateful for the hospitality.
The final version of the paper was prepared at the Institut de Math\'ematiques de Jussieu; we thank Elisha Falbel and the AGN \emph{Structures G\'eom\'etriques \& Triangulations} for hosting us. 
We are very grateful to Nathan Dunfield and Robert Bieri for several very helpful comments and suggestions. The first author was supported by the SFB 1085 `Higher invariants', funded by the Deutsche Forschungsgemeinschaft (DFG). The second author is partially supported under the Australian Research Council's Discovery funding scheme (project number DP140100158). 

\section{The marked polytope of a $(2,1)$--presentation}
\label{section:twoonepres}

\subsection{Marked polytopes}

Let $V$ be a real vector space and $Q=\{Q_1,\dots,Q_k\}\subset V$ be a finite  set. The  \emph{$($convex$)$ hull of $Q$} is the set
\[ \PP(Q)=\hull(Q)=\left\{ \sum_{i=1}^k t_iQ_i\,\left|\ \sum_{i=1}^k t_i= 1, \ t_i\geq 0\right.\right\}.\]
A \emph{polytope} in $V$ is a subset of $V$ which is the hull of a finite non-empty subset of $V$. For any polytope $\PP$ there exists a unique smallest subset $\VV(\PP)\subset \PP,$ such that $\PP$ is the hull of $\VV(\PP)$.
 The elements of $\VV(\PP)$ are called the \emph{vertices} of $\PP$. 

A \emph{marked polytope} is a polytope together with a (possibly empty) set of marked vertices. 
Given a finite multiset $Q=[Q_1,\dots,Q_k]\subset V$, we  denote by $\MM(Q)$ the polytope $\PP(Q)$, where we mark each vertex $\MM(Q)$ that has multiplicity precisely  one in $Q.$

\subsection{The Minkowski sum of marked polytopes}
\label{section:minkowski}

Let $V$ be a real vector space and let $\PP$ and $\QQ$ be two polytopes in $V$. The \emph{Minkowski sum of $\PP$ and $\QQ$}
is defined as the set
\[ \PP+\QQ:=\{ p+q\,|\, p\in \PP\mbox{ and }q\in \QQ\}.\]
It is straightforward to see that $\PP+\QQ$ is again a polytope. Furthermore, for each vertex $u$ of $\PP+\QQ$
there exists a unique vertex $v$ of $\PP$ and a unique vertex $w$ of $\QQ$ such that $u=v+w$. Conversely, for each vertex $v$ of $\PP$ there exists a (not necessarily unique) vertex $w$ of $\QQ$ such that $v+w$ is a vertex of $\PP+\QQ$.

If $\PP,\QQ$ and $\RR$ are polytopes with $\PP+\QQ=\RR,$ then we write $\PP=\RR-\QQ$.
Note that 
\[ \PP=\{p\in V\,|\, p+\QQ\subseteq \RR\},\]
in particular given polytopes $\QQ$ and $\RR,$ if  the polytope $\RR-\QQ$ exists, then it is well-defined.

If $\MM$ and $\NN$  are two marked polytopes, then we define the \emph{$($marked$)$ Minkowski sum of $\MM$ and $\NN$} as the Minkowski sum $\MM+\NN$ with set of marked vertices precisely those that are the sum of a marked vertex of $\MM$ and a marked vertex of $\NN$. An example is given in Figure~\ref{fig:summp}.

\begin{figure}[h]
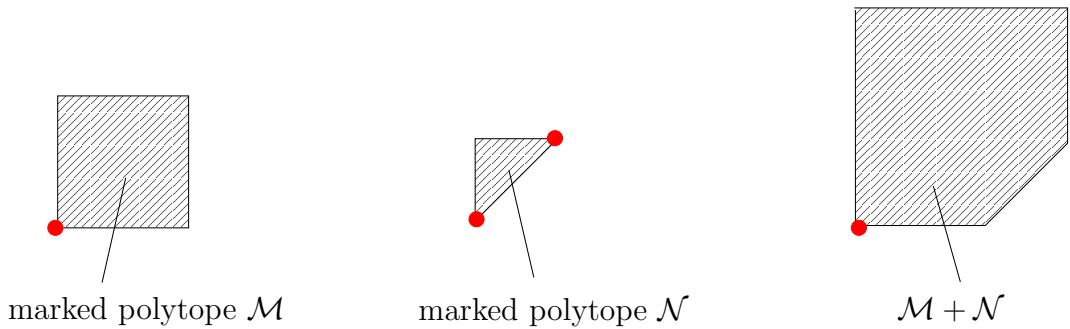
\caption{Example of the Minkowski  sum of two marked polytopes.}\label{fig:summp}
\end{figure}

Now we consider marked polytopes in $\R^2$ in more detail:
\bn
\item We denote by $\XX=[0,1]\times \{0\}$ (resp.\thinspace $\YY=\{0\}\times [0,1]$) the marked polytope in $\R^2$ with both vertices marked. This is a horizontal (resp.\thinspace vertical) interval of length one with marked endpoints. 
\item Given a polytope $\PP$ in $\R^2$ we let $x_0(\PP)$ be the minimal $x$-coordinate of any point in $\PP$ and 
 $x_1(\PP)$ be the maximal $x$-coordinate of any point in $\PP$. The definition of $y_0(\PP)$ and $y_1(\PP)$ is completely analogous.
\item We denote by $x_0^0(\PP)$ (resp.\thinspace $x_0^1(\PP)$) the points on the vertical $x_0$-slice $\PP\cap \{x_0\}\times \R$ of  $\PP$ with minimal (resp.\thinspace maximal) $y$-value.
Similarly define  $x_1^0(\PP)$ and $x_1^1(\PP),$ as well as $y_i^k(\PP)$ with the roles of the $x$ and $y$--coordinates reversed.
All the resulting points are vertices of $\PP$.
We refer to Figure \ref{fig:corners} for an illustration.
\en

\begin{figure}[h]
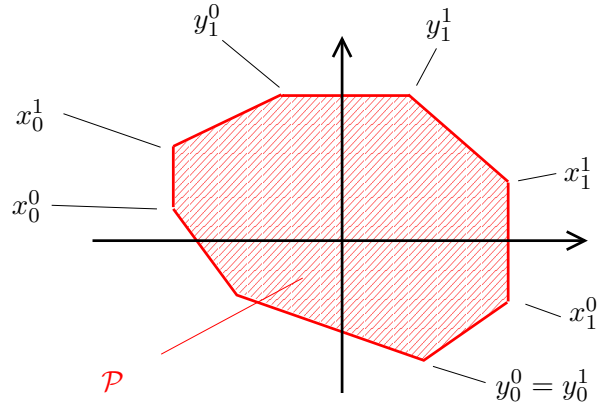
\caption{The corner points $x_i^j(\PP)$ and $y_i^j(\PP)$.}\label{fig:corners}
\end{figure}

\begin{lemma}\label{lem:subtracty}
Let $\NN$ be a marked polytope in $\R^2$. Suppose that for $i=0,1$ the following two conditions are satisfied:
\bn
\item the difference in the $y$-coordinates of $x_i^0(\NN)$ and $x_i^1(\NN)$ is at least one,
\item if the difference in the $y$-coordinates of $x_i^0(\NN)$ and $x_i^1(\NN)$ is precisely one, then either both
$x_i^0(\NN)$ and $x_i^1(\NN)$ are marked or both are not marked.
\en
Then there exists a unique marked polytope $\MM$ with $\MM+\YY=\NN$. 
\end{lemma}

The lemma is an elementary exercise in polytope theory, we therefore merely outline the proof. 

\begin{proof}[Proof $($Sketch$)$.]
Throughout the proof we refer to Figure \ref{fig:subtracty} for an illustration.
For $i,j\in\{0,1\}$ we write $x_i^j=x_i^j(\NN)$ and $y_i^j=y_i^j(\NN)$.
\begin{figure}[h]
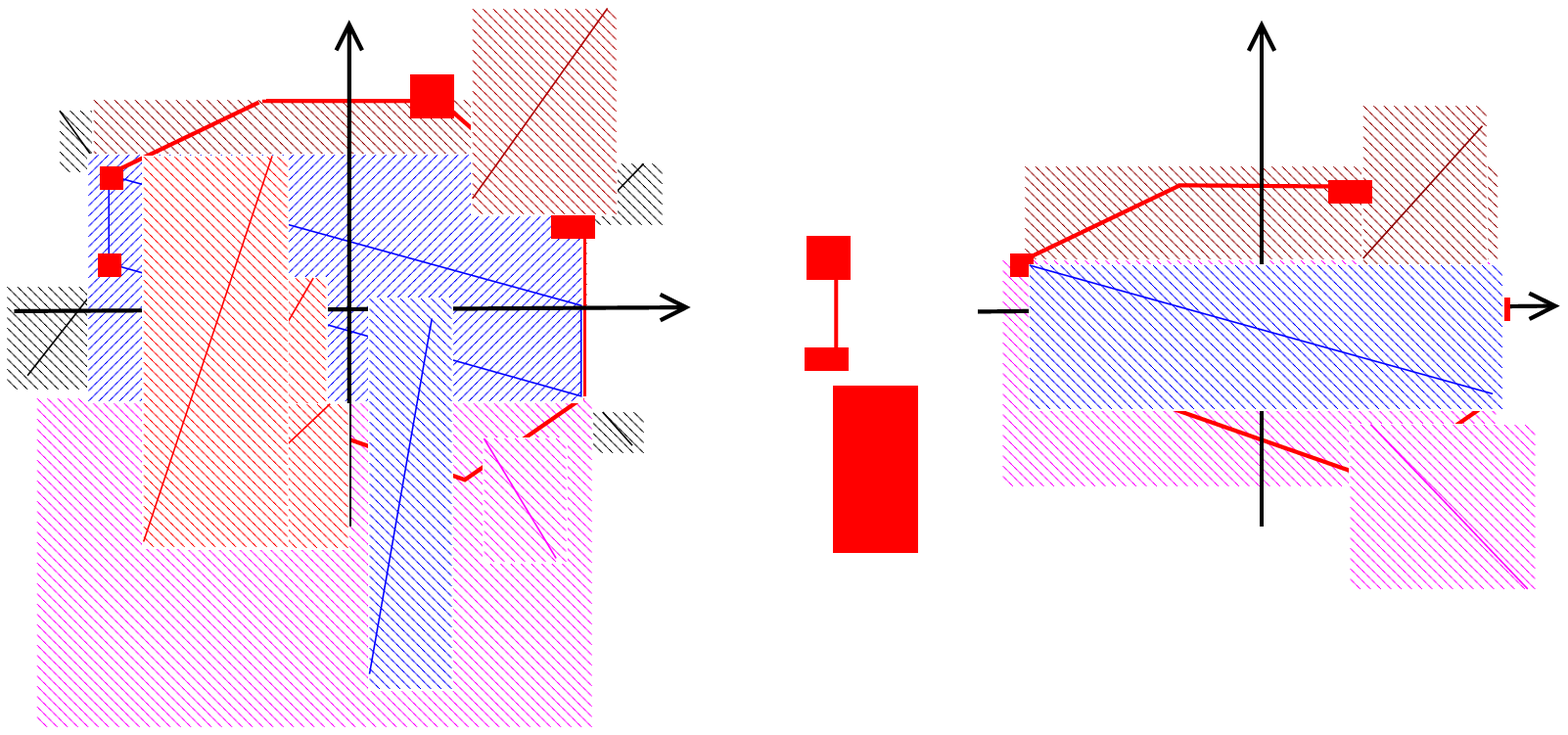
\caption{Subtracting $\YY$.}\label{fig:subtracty}
\end{figure}

We consider the parallelogram $\PP=\hull(x_0^0,x_0^0+(0,1),x_1^0,x_1^0+(0,1))$.
It follows from the assumptions that $\PP$ is contained in $\NN.$ First suppose that the closure of its complement in $\NN$ consists of two polytopes. Denote $\PP_0$ the polytope below $\PP,$ and $\PP_1$ the polytope above $\PP$.

We denote by $\PP_1'$ the polytope obtained by translating $\PP_1$ down by one, and let $\MM$ be the union of $\PP_0$ and $\PP_1'$.
It is straightforward to verify that as polytopes without marking, we have $\MM+\YY=\NN$ and that $\MM$ is the only polytope which has this property.

It remains to mark the appropriate vertices of $\MM.$
For each vertex of $\MM$ there exists a vertex of $\YY$ such that the sum is a vertex of $\NN$.
Mark the vertex of $\MM$ if and only if the vertex of $\NN$ is marked. Using the second hypothesis, it follows that this marking of $\MM$ is well-defined, i.e.\ independent of the choice of the vertex of $\YY$, and that it is the only marking for $\NN$ which has the desired property.

This concludes the generic case. In the degenerate cases, either $\NN=\PP$ or the complement of $\PP$ consists of a single polytope and it is easy to adjust the above arguments.

Finally the uniqueness of $\MM$ is straightforward to verify, we leave this to the reader.
\end{proof}

\begin{corollary}\label{cor:subtractxy}
Let $\NN$ be a marked polytope in $\R^2$. We suppose that for $i=0,1$ the following conditions are satisfied:
\bn
\item  the difference in the $y$-coordinates of $x_i^0(\NN)$ and $x_i^1(\NN)$ is at least one,
\item if  the difference in the $y$-coordinates of $x_i^0(\NN)$ and $x_i^1(\NN)$ is precisely one, then either both
$x_i^0(\NN)$ and $x_i^1(\NN)$ are marked or both are not marked.
\item the difference in the $x$-coordinates of $y_i^0(\NN)$ and $y_i^1(\NN)$ is at least one,
\item if  the difference in the $x$-coordinates of $y_i^0(\NN)$ and $y_i^1(\NN)$ is precisely one, then either both
$y_i^0(\NN)$ and $y_i^1(\NN)$ are marked or both are not marked.
\en
Then there exists a unique marked polytope $\MM$ with $\MM+\XX+\YY=\NN$. 
\end{corollary}

\begin{proof}
By our assumptions (1) and (2) we can apply Lemma~\ref{lem:subtracty} to $\NN$. It is straightforward to see that properties (3) and (4) are preserved and we can apply the obvious
version of Lemma~\ref{lem:subtracty} for subtracting $\XX$ instead of $\YY$. It is once again easy to see that the resulting marked polytope is unique.
\end{proof}

\subsection{The marked polytope of a nice $(2,1)$--presentation}
\label{section:defpolytopepi}

Throughout the paper, given a nice $(2,1)$--presentation $\pi=\ll x,y\,|\,r\rr$ we adopt the following notation:
\bn
\item We denote by $l(r)$ the length of $r$ and given $i\in \{0,\dots,l(r)\}$ 
we denote by $r_i$ the product of the first $i$ letters appearing in $r$.
More precisely,  we  write $r=g_1g_2\dots g_{l(r)}$ with $g_1,\dots,g_{l(r)}\in \{x^{\pm 1},y^{\pm 1}\}$, and given $i\in \{0,\dots,l(r)\}$ we  define $r_i:=g_1\cdot \dots \cdot g_i$. 
\item We denote $\eps\co G_\pi\to H_1(G_\pi;\Z)$ the obvious map
 and we view $H_1(G_\pi;\Z)\cong \Z^2$ as a subset of $H_1(G_\pi;\R)$.
Note that $\eps(x)$ and $\eps(y)$ give rise to a basis for $H_1(G_\pi;\R)$ which we will sometimes use to identify $H_1(G_\pi;\Z)$ with $\Z^2$ and  we will use it to identify $H_1(G_\pi;\R)$ with $\R^2$.
\item Given a finite multiset $[g_1,\dots,g_k]$ of elements in $G_\pi$, let $\MM(g_1,\dots,g_k)$ be the marked polyhedron $\MM([\eps(g_1),\dots,\eps(g_k)])$
 in $H_1(G_\pi;\R)$. A vertex $v$ is marked if there is precisely one $g_i$ with $\eps(g_i)=v$.
\en

Now we have the following lemma.

\begin{lemma}\label{lem:defmpi}
Let $\pi=\ll x,y\,|\,r\rr$ be a nice $(2,1)$--presentation. We write
$\NN=\MM(r_0,\dots,r_{l(r)})$.
Then there exists a unique marked polytope $\MM$ in $H_1(G_\pi;\R)=\R^2$ with 
\[ \MM+\XX+\YY=\NN.\]
\end{lemma}

In the following we denote by $\MM_\pi$ the  marked polytope of Lemma~\ref{lem:defmpi}.

\begin{proof}
We will prove the lemma by verifying  that the conditions of  Corollary~\ref{cor:subtractxy} are satisfied. 

We write $l=l(r)$ and for $i,j\in\{0,1\}$ we write $x_i^j=x_i^j(\NN)$ and $y_i^j=y_i^j(\NN)$.
Note that $l\geq 1$ since we assumed that $r$ is not the empty word.
Now we view the indices for the $r_i$'s as being elements in $\Z_{l(r)}$. 
Given $i\in \Z_{l}$  we say that \emph{the step at $i$ is horizontal} if 
$\eps(r_{i+1})-\eps(r_i)=(\pm 1,0)$. Similarly we define a vertical step.
We make the following observations:
\bn
\item[(a)] For each $i\in \Z_{l}$ the step is either horizontal or vertical.
\item[(b)] Since $r$ is cyclically reduced we have $\eps(r_{i+2})\ne \eps(r_i)$ for any $i$.
\en

\begin{claim}
Let $i\in \Z_l$. 
\bn
\item If $\eps(r_i)=x_0^0$, then either 
$\eps(r_{i-1})=x_0^0+(0,1)$ or $\eps(r_{i+1})=x_0^0+(0,1)$.
\item If $\eps(r_i)=x_0^1$, then either 
$\eps(r_{i-1})=x_0^0+(0,-1)$ or $\eps(r_{i+1})=x_0^1+(0,-1)$.
\en
\end{claim}

We only prove the first statement, the other statement is proved exactly the same way.
If the step at $i$ is vertical, then it follows from the definition of $x_0^0$ and from (a) that $\eps(r_{i+1})=x_0^0+(0,1)$.
If the step at $i$ is horizontal, then by the definition of $x_0^0$ we have $\eps(r_{i+1})=x_0^0+(1,0)$.
By (b) we now see that the step at $i-1$ is vertical, which by the definition of $x_0^0$ implies that $\eps(r_{i-1})=x_0^0+(0,1)$.
This concludes the proof of the claim.

It  follows immediately from the definitions that $\NN$ satisfies conditions (1) and (2) of Corollary~\ref{cor:subtractxy} for $i=0$.
Exactly the same argument shows that the conditions are satisfied for $i=1$, and that also conditions (3) and (4) are satisfied.
The lemma is thus a consequence of Corollary~\ref{cor:subtractxy}.
\end{proof}

If  $\pi=\ll x,y\,|\,r \rr$  is a nice $(2,1)$--presentation and if $r'$ is a cyclic permutation of the word $r$, then
$\pi'=\ll x,y\,|\,r' \rr$ is also a nice presentation  which presents the same group.
Now we will relate $\MM_\pi$ and $\MM_{\pi'}$.

\begin{lemma}\label{lem:polytopetranslate}
Let $\pi=\ll x,y\,|\,r \rr$ be a nice $(2,1)$--presentation.
Let $r'$ be a cyclic permutation of $r$.  We denote by 
$\pi'=\ll x,y\,|\,r' \rr$ the corresponding presentation.
Then $\MM_{\pi'}$ differs from $\MM_\pi$ by a translation by a vector in $H_1(G_\pi;\Z)$. 
\end{lemma}

\begin{proof}
We write $l=l(r)=l(r')$. It is straightforward to see that
$\MM(r'_1,\dots,r'_l)$ is a translate of $\MM(r_1,\dots,r_l)$ by a vector in $H_1(G_\pi;\Z)$. 
The lemma is an immediate consequence of this observation.
\end{proof}

We conclude this section with the following elementary lemma. We will not make use of it in the paper and we leave the proof to the reader.

\begin{lemma}
Given any marked polytope $\MM$ in $\R^2$ with integer vertices there exists a nice $(2,1)$-presentation $\pi$ with $\MM=\MM_\pi$.
\end{lemma} 

\subsection{Relation of the two definitions of $\mathbf{\MM_\pi}$}
Let $\pi=\ll x,y\,|\,r\rr$ be a nice $(2,1)$--presentation.
We sketched a definition for $\MM_\pi$  in the introduction and using a somewhat different language we gave a more rigorous definition in
in Section~\ref{section:defpolytopepi}.

We obtained both polytopes (without the marking) by the following process:
\bn
\item we first consider the polytope given by  the points $\eps(r_0),\dots,\eps(r_{l(r)})$ in $H_1(G_\pi;\R)=\R^2$,
\item we then shrink the polytope by one in both the $x$-direction and the $y$-direction.
\en
In the introduction we were a little vague in how to assign markings, the argument in Section~\ref{section:defpolytopepi} shows that this can be done in a coherent way.  The task of spelling out the details of why the two definitions are the same is left to the reader.

\section{Interpretation of $\MM_\pi$ in terms of Fox derivatives} 
\label{sec:marked polytope via fox calculus}
In this section we will interpret the marked polytope $\MM_\pi$ in terms of Fox derivatives. This point of view will be crucial in our proofs.

\subsection{The marked polytope for elements of group rings}

Let $G$ be a group.
Throughout the paper, given $f\in \Z[G]$  and given $g\in G$ we denote by $f_g$ the $g$-coefficient of $f$.

We write $V=H_1(G;\R)$ and we denote by $\eps\co G\to V$ the canonical map. 
Given $f\ne 0\in \Z[G]$ we  refer to 
\[ \PP(f):=\PP\left(\{g\,|\,  g\in G\mbox{ with }f_g\ne 0\}\right)\subset V\]
as the \emph{polytope of $f$}. We consider the multiset
 $[|f_g|\cdot g\,|\, g\in G]$ where the notation $|f_g|\cdot g$ means that  $g\in G$ appears $|f_g|$-many times in the multiset. Then we refer to
\[ \MM(f):=\MM\big([\,|f_g|\cdot g\,|\, g\in G]\big)\subset V\]
as the \emph{marked polytope of $f$}. 
We will also need the following definitions.
\bn
\item 
For $v\in V$ we refer to 
\[ f^v:=\sum_{g\in \eps^{-1}(v)}f_gg \]
as the \emph{$v$-component of $f$}.
\item We say that an element $r\in \Z[G]$ is a monomial if it is of the form  $r=\pm g$ for some $g\in G$.
\en
Now we can  formulate the following alternative definition of the marking of the marked polytope $\MM(f)$.

\begin{lemma}\label{lem:defmf}
Let $G$ be a group and let $f\ne 0\in \Z[G]$. A vertex $v$ of $\MM(f)$ is marked if and only if $f_v$ is a monomial.
\end{lemma}

We will later on need the following lemma.

\begin{lemma}\label{lem:productadd}
Let $G$ be a group and let $f,g\in \Z[G]$. Then the following hold:
\bn
\item If for every vertex $v$ of $\PP(f)$ the element $f^v\in \Z[G]$ is not a zero divisor, then 
\[ \PP(f\cdot g)=\PP(f)+\PP(g).\]
\item If each vertex  of $\MM(f)$ is marked, then
\[ \MM(f\cdot g)=\MM(f)+\MM(g).\]
\en 
\end{lemma}

\begin{proof}
\bn
\item
Let $v$ be a  vertex of $\PP(f)$ and $w$ be a vertex of $\PP(g)$. By assumption it follows that $f^v\cdot g^w\ne 0$.
It follows easily from the definitions that $\PP(f\cdot g)=\PP(f)+\PP(g)$.
\item We will use the characterization of marked vertices given by Lemma~\ref{lem:defmf}. Since monomials are not zero divisors it follows from (1)  that $\PP(f\cdot g)=\PP(f)+\PP(g)$. 
Furthermore, our assumptions on $f$  imply that for any vertex $v$ of $\MM(f)$ and any vertex $w$ of $\MM(g)$ the product 
$f^v\cdot g^w$ is a monomial if and only if $g^w$ is a monomial. 
The statement on marked polytopes  again follows easily from the definitions.
\en
\end{proof}

\subsection{Fox calculus}

In the following we denote by $F$ the free group with generators $x_1,\dots,x_k$.
We  denote by  $\frac{\partial }{\partial x_i}\co \Z[F]\to \Z[F]$ the \emph{Fox derivative with respect to $x_i$}, i.e.\ the unique $\Z$-linear map such that
\[ 
\frac{\partial x_i}{\partial x_i}=1,\quad \frac{\partial x_j}{\partial x_i}=0\mbox{ for $i\ne j$ and with } \frac{\partial uv}{\partial x_i}=\frac{\partial u}{\partial x_i}+u\frac{\partial v}{\partial x_i} \mbox{ for all  $u,v\in F$.}\]
We refer to \cite{Fo53} for details and more information on Fox derivatives.
In the following, given $u\in \Z[F]$ we often write 
\[ u_{x_i}=\frac{\partial u}{\partial x_i}.\]
We denote by $\alpha\co \Z[F]\to \Z$ the augmentation map which is the unique $\Z$-linear map with $\alpha(x_i)=1$ for $i=1,\dots,k$.
The fundamental formula for Fox derivatives (see \cite[p.~551]{Fo53}) says that for any $f\in \Z[F]$ we have
\[  f-\alpha(f)\cdot e =\sum_{i=1}^kf_{x_i}(x_i-1)\]
where $e$ denotes the trivial element in $F$.
For example, if $\pi=\ll x,y|r\rr$ is a $(2,1)$--presentation, then 
\[ r-\alpha(r)\cdot e=r_x(x-1)+r_y(y-1)\in \Z[\ll x,y\rr].\]
But $\alpha(r)=1$ since $r$ is  a word in $x$ and $y$. Furthermore $r=e\in G_\pi$. We thus see that
\be \label{equ:foxzero} r_x(x-1)=-r_y(y-1)\in \Z[G_\pi].\ee

\subsection{Fox derivatives and 1-relator groups}
The following theorem is due to Weinbaum \cite{We72} (see also \cite[Proposition~II.5.29]{LS77}).

\begin{theorem}\label{thm:weinbaum}
Let  $ \pi=\ll x_1,\dots,x_k\,|\, r\rr$
be  a presentation  where $r$ is a cyclically reduced word.
If $w$ is a proper, non-empty subword of $r$, then $w$ represents a non-trivial element in $G_\pi$.
\end{theorem}

\begin{corollary}\label{cor:distinct}
Let  $ \pi=\ll x_1,\dots,x_k\,|\, r\rr$
be  a presentation. 
If $r$ is cyclically reduced, then the summands in
\[ \frac{\partial r}{\partial x_i}=\sum_{j=1}^s \eta_j w_j\]
represent distinct elements in $\Z[G_\pi]$.
\end{corollary}

\begin{proof}
We write $ r=x_{m_1}^{\eps_1}x_{m_2}^{\eps_2}\cdot \dots \cdot x_{m_l}^{\eps_l}$ with $\eps_1,\dots,\eps_l\in \{-1,1\}$..
Let  $i\in \{1,\dots,k\}$. We denote by $s$ the number of times $x_i$ appears among $x_{m_1},\dots,x_{m_l}$.
It  follows immediately from the definition of the Fox derivative
 that  there exist $\eta_1,\dots,\eta_s\in \{-1,1\}$ and $0\leq n_1<n_2<\dots<n_s\leq l$ such that
\[ \frac{\partial r}{\partial x_i}=\sum_{j=1}^s \eta_j w_j\]
where for $j=1,\dots,s$ the element $w_j$ is represented by the subword of $r$ consisting of the first $n_j$ letters appearing in $r$, i.e.
\[ w_j=x_{m_1}^{\eps_1}x_{m_2}^{\eps_2}\cdot \dots \cdot x_{m_{n_j}}^{\eps_{n_j}}.\]
The words $w_1,\dots,w_s$ differ by a proper, non-empty subword of $r$.
Thus the desired statement follows from Theorem~\ref{thm:weinbaum}.
\end{proof}

\subsection{Fox derivatives and the marked polytope for a nice $(2,1)$--presentation}
In this section, given a nice $(2,1)$--presentation $\pi=\ll x,y\,|\,r \rr$ we will  express the marked polytope $\MM_\pi$ in terms of the Fox derivatives 
$r_x$ and $r_y$. Throughout this section we will several times  make use of the observation 
that $\MM(x-1)=\XX$ and $\MM(y-1)=\YY$. 

\begin{proposition}\label{prop:foxpolytope}
Let $\pi=\ll x,y\,|\,r \rr$
be a nice $(2,1)$--presentation. Then 
\[ \MM(r_y)=\MM_\pi+\MM(x-1)\mbox{ and }\MM(r_x)=\MM_\pi+\MM(y-1).\]
\end{proposition}

\begin{proof}
We will only prove that $\MM(r_y)=\MM_\pi+\MM(x-1)$. The other equality is proved completely analogously. 

We denote by $\NN$ the marked polytope we introduced in Lemma~\ref{lem:defmpi} which is given by tracing out the word $r$. Recall that $\MM_\pi$ is the unique marked polytope with $\NN=\MM_\pi+\MM(x-1)+\MM(y-1)$. 
By  Lemma~\ref{lem:subtracty} it thus suffices to show that 
$\MM(r_y)+\MM(y-1)=\NN=\MM_\pi+\MM(x-1)+\MM(y-1)$.

In a certain sense it is obvious that $\MM(r_y)+\MM(y-1)=\NN$. Indeed, this follows from the observation that $\MM(r_y)+\MM(y-1)$ is given by all the vertical
edges traced out in the definition of $\NN$. We thus obtain the same marked polytope.
The remainder of this proof is taken up by making this observation rigorous. 

First we note that it is straightforward to verify that if the statement holds for some $r$, then it also holds for any cyclic permutation of $r$. 
Therefore we can  take a cyclic permutation  of $r$ such that the resulting relator starts with  $x$ or $x^{-1}$. Without loss of generality we can thus assume that  $r=x^{m_1}y^{n_1}\cdot \dots \cdot x^{m_k}y^{n_k}$ where all the $m_i$ and $n_i$ are non-zero.

Now we  note that 
\[ \NN=\MM\left(\bigcup_{i=1}^k \bigcup_{j=0}^{m_i-1} x^{m_1}y^{n_1}\cdot \dots \cdot x^{m_{i-1}}y^{n_{i-1}}x^j\,\,\cup\,\,
\bigcup_{i=1}^k \bigcup_{j=0}^{n_i-1} x^{m_1}y^{n_1}\cdot \dots \cdot x^{m_{i-1}}y^{n_{i-1}}x^{m_i}y^j\right).\]
As we are taking the convex hull we can leave out points which lie in the interior of a segment connecting two other points. Thus we have
\[ \NN=\MM\left(\bigcup_{i=1}^k x^{m_1}y^{n_1}\cdot \dots \cdot x^{m_{i-1}}y^{n_{i-1}}x^{m_i}\,\,\cup\,\,
\bigcup_{i=1}^kx^{m_1}y^{n_1}\cdot \dots \cdot x^{m_{i-1}}y^{n_{i-1}}x^{m_i}y^{n_i}\right).\]

Now we turn to $\MM(r_y)+\MM(y-1)$.
We first note that for any $n\ne 0\in \Z$ we have 
\[ \frac{\partial (y^n)}{\partial y}\cdot (y-1)=y^n-1.\]
It  follows from this observation and from Lemma~\ref{lem:productadd} that 
\[ \ba{rcl} \MM(r_y)+\MM(y-1)&=&\MM(r_y\cdot (y-1))\\
&=&\MM\left(\left(\sum_{i=1}^k x^{m_1}y^{n_1}\cdot \dots \cdot x^{m_{i-1}}y^{n_{i-1}}x^{m_i}\frac{\partial (y^{n_i})}{\partial y}\right)(y-1)\right)\\[2mm]
&=&\MM\left(\sum_{i=1}^k x^{m_1}y^{n_1}\cdot \dots \cdot x^{m_{i-1}}y^{n_{i-1}}x^{m_i}\left(y^{n_i}-1 \right)\right).\ea \]
The same argument as in the proof of Corollary~\ref{cor:distinct} shows that all the summands are pairwise different t in $G_\pi$. Thus it follows from Lemma~\ref{lem:defmf} that 
\[ \ba{rcl} &&\MM\left(\sum_{i=1}^k x^{m_1}y^{n_1}\cdot \dots \cdot x^{m_{i-1}}y^{n_{i-1}}x^{m_i}\left(y^{n_i}-1 \right)\right)\\[2mm]
&=&\MM\left(\bigcup_{i=1}^k x^{m_1}y^{n_1}\cdot \dots \cdot x^{m_{i-1}}y^{n_{i-1}}x^{m_i}y^{n_i}\,\,\cup\,\,
\bigcup_{i=1}^kx^{m_1}y^{n_1}\cdot \dots \cdot x^{m_{i-1}}y^{n_{i-1}}x^{m_i}\right),\ea\]
but this is precisely $\NN$. Thus we showed that $\MM(r_y)+\MM(y-1)=\NN$.
\end{proof}

\section{The proof of Theorem~\ref{mainthm}}
\label{sec:proofs of first theorem}

\subsection{Basic properties of the Bieri--Neumann--Strebel invariant}
\label{section:bns}

Let $G$ be a finitely generated group. The Bieri--Neumann--Strebel \cite{BNS87} invariant $\S(G)$ of $G$ is by definition a subset of $S(G):=(\hom(G,\R)\sm \{0\})/\R_{>0}$. We refer to \cite{BNS87} for the precise definition,
but in order to give a flavor of the invariant we recall three properties:
\bn
\item An epimorphism $\phi\in \hom(G,\Z)$ represents 
an element in $\S(G)$ if and only if it corresponds to an ascending HNN-extension. More precisely, if and only if there exists an isomorphism
\[ f\co G\to \ll A,t\,|\, A=t^{-1}\varphi(A)t\rr\]
where $A$ is a finitely generated group and $\varphi\co A\to A$ is a monomorphism, such that $\phi$ corresponds under $f$ to the epimorphism given by $t\mapsto 1$ and $a\mapsto 0$ for $a\in A$.\\
At this point it is perhaps worth pointing out that at times in the literature an HNN-extension of the form $\ll A,t\,|\, A=t\varphi(A)t^{-1}\rr$ is also referred to as an ascending  HNN-extension. Nonetheless, it follows from the discussion on \cite[p.~456]{BNS87} and the definition of ascending HNN-extension on \cite[p.~465]{BNS87} that our definition of ascending HNN-extension matches the definition of \cite{BNS87}.
\item A homomorphism $\phi\in \hom(G,\Z)$ has the property that $\phi$ and $-\phi$ represent elements in $\S(G)$
if and only if $\ker(\phi)$ is finitely generated.
\item $\Sigma(G)$ is an open subset of $S(G)$.
\en
Here the first two properties follow from \cite[Proposition~4.3]{BNS87} (see also \cite[Corollary~3.2]{Brn87}) and the third one is \cite[Theorem~A]{BNS87}.

\subsection{Twisted homology groups}\label{section:twihom}
Let $X$ be a finite CW-complex with $G=\pi_1(X)$. We denote by $\wti{X}$ the universal cover of $X$. 
The deck transformation group $G$ acts on the left on $\wti{X}$. Therefore the chain complex $C_*(\wti{X})$ is a chain complex of free left $\Z[G]$-modules.

If $R$ is a ring and $M$ is a $(R,\Z[G])$-bimodule, then consider the chain complex
\[ C_*(X;M)=M\otimes_{\Z[G]} C_*(\wti{X}) \]
of left $R$-modules and the corresponding twisted homology groups 
$H_*(X;M)$ which are also left $R$-modules.

\subsection{The chain complex corresponding to a presentation}
\label{section:2complexpi}
Given a presentation $\pi=\ll x_1,\dots,x_k\,|\,r_1,\dots,r_l\rr$ we  denote by $X_\pi$ the corresponding CW-complex with one 0-cell, $k$ 1-cells corresponding to the generators
and $l$ 2-cells corresponding to the relators.
With appropriate lifts of the cells of $X$ to the universal cover $\wti{X}$ the complex $C_*(\wti{X})$ is then given by
\[ 0\to \Z[G_\pi]^l\xrightarrow{\bp \tfrac{\partial r_1}{\partial x_1}&\dots &\tfrac{\partial r_1}{\partial x_k}\\
\vdots &&\vdots \\
\tfrac{\partial r_r}{\partial x_1}&\dots & \tfrac{\partial r_r}{\partial x_k}\ep }\Z[G_\pi]^k\xrightarrow{\bp x_1-1\\\dots \\x_k-1\ep}\Z[G_\pi]\to 0.\]
Here we recall that we always view vectors as row-vectors and that we multiply by matrices \emph{on the right}. This (somewhat confusing) convention is forced on us by the fact that we consider left-modules.

\subsection{Generalized Novikov-homology}
In the following, given a group $G$ and $\phi\in \hom(G,\R)$ we  consider 
\[ \what{\Z[ G]_\phi}:=\left\{ \sum_{g\in G}f_gg\,\left|\, \ba{l}\mbox{for every $C\in \R$ there exist only finitely many $g\in G$}\\
\mbox{with $\phi(g)>C$ and $f_g\ne 0$,}\ea \right.\right\}\]
 the Sikorav-Novikov completion \cite{No81,Si87} of  the group ring $\Z[G]$ with respect to $\phi$.
It is straightforward to verify that $\what{\Z[ G]_\phi}$ is indeed a ring with the obvious addition and the `naive' multiplication.

Given $f\in \Z[G]$ we define
\[ T_\phi(f):=\sum_{g\in G, \phi(g)=m} f_gg\]
where $m:=\min\{ \phi(g)\,|\, f_g\ne 0\}$.  Furthermore, given a ring $R$ and  $r\in R$ we say that $s\in R$ is an \emph{left-inverse to $r$} if $sr=1$.
We recall the following well-known lemma. We leave the straightforward proof to the reader.

\begin{lemma}\label{lem:units}
Let $G$ be a group and let $\phi\colon G\to \R$ be a homomorphism.
Let $f\in \Z[G]$. 
If $f$  has a left-inverse in   $\what{\Z[ G]_\phi}$, then $T_\phi(f)$ has a left-inverse in  $\Z[G]$. Conversely,
if $T_\phi(f)$ is a monomial,  then $f$ has a left-inverse in  $\in \what{\Z[ G]_\phi}$.
\end{lemma}

A group $G$ is called \emph{locally indicable} if any finitely generated non-trivial subgroup of $G$ admits an epimorphism onto $\Z$. 
If $G$ is locally indicable, then the proof of Theorem~13 in  \cite{Hi40}  shows  that monomials are  the only elements in  $\Z[G]$ that have a left-inverse. 
We thus obtain the following variation on Lemma~\ref{lem:units}

\begin{lemma}\label{lem:unitslocallyindicable}
Let $G$ be a locally indicable group and let $\phi\colon G\to \R$ be a homomorphism.
Then $f\in \Z[G]$  has a left-inverse in   $\what{\Z[ G]_\phi}$ if and only 
if $T_\phi(f)$ is a monomial.
\end{lemma}

One of the key ingredients in the proof of Theorem~\ref{mainthm} is the 
 following theorem of Sikorav.

\begin{theorem}\label{thm:sikorav}
Given a group $G$ a non-zero homomorphism $\phi\in \hom(G,\R)$ represents an element in $\S(G)$ if and only if 
\[ H_0(G;\what{\Z[ G]_\phi})=0 \mbox{ and }  H_1(G;\what{\Z[ G]_\phi})=0.\]
\end{theorem}

\begin{proof}
Given a finitely generated group $G$ Bieri-Renz \cite{BR88} introduce an invariant $\Sigma(G,\Z)$
that is also a subset of $S(G)$. Let $\phi\in \hom(G,\R)$ be a  non-zero homomorphism $\phi\in \hom(G,\R)$.
Then the following two statements hold:
\bn
\item By \cite[p.~465]{BR88} the homomorphism $\phi$ represents an element in $\Sigma(G,\Z)$ if and only if $-\phi$ represents an element in $\Sigma(G)$.
\item 
The statements of  \cite[p.~86]{Si87}, \cite[p.~953]{Bi07} and \cite[Section~3]{FGS10} imply that $\phi$
 represents an element in $\S(G,\Z)$ if and only if 
\[ H_0(G;\what{\Z[ G]_{-\phi}})=0 \mbox{ and }  H_1(G;\what{\Z[ G]_{-\phi}})=0.\]
\en
Together these two statements imply Theorem \ref{thm:sikorav}.
\end{proof}

The definitions of the Bieri-Neumann-Strebel invariant, the Bieri-Renz invariant and generalized Novikov homology involve various choices and conventions. In order to make sure that the signs are correct as stated in the proof above we consider the Baumslag-Solitar group 
\[ B=\ll a,t| t^{-1}a^2ta^{-1}\rr\] 
with $\phi(t)=1$ und $\phi(a)=0$. As we have seen in Section \ref{section:bns},  $\phi$ corresponds to an ascending HNN-extension, so in particular $\phi\in \Sigma(B)$.  We refer again to the discussion in Section~\ref{section:bns} for the definition of ascending HNN-extension as in \cite{BNS87} and the relationship to the invariant $\Sigma(\pi)$.
In this case $r_a=t^{-1}(1+a)-1$. 
An argument similar to the one provided in the proof of Theorem~\ref{mainthm}
shows that $r_a$ is invertible in $\what{\Z[ G]_\phi}$ but it is not invertible in
$\what{\Z[ G]_{-\phi}}$, which then implies that  $H_1(G;\what{\Z[ G]_\phi})=0$ but 
 $H_1(G;\what{\Z[ G]_{-\phi}})\ne 0$.

\subsection{The  proof of Theorem~\ref{mainthm}}

In this section we will finally give the proof of Theorem~\ref{mainthm}.

\begin{proof}[Proof of Theorem~\ref{mainthm}]
Let $\pi=\ll x,y\,|\, r\rr$ be a nice $(2,1)$--presentation. We write $G=G_\pi$.
Let $\phi\in \hom(G,\R)$ be a  non-zero homomorphism. 
 It follows from the discussion in Section~\ref{section:2complexpi} that  the chain complex
$\what{\Z[ G]_\phi} \otimes_{\Z[G]} C_*(\wti{X}) $ is given by
\[ 0\to \what{\Z[ G]_\phi}\underset{\partial_2}{\xrightarrow{\bp r_x& r_y\ep}}\what{\Z[ G]_\phi}^2\underset{\partial_1}{\xrightarrow{\bp x-1\\y-1\ep}}\what{\Z[ G]_\phi}\to 0.\]

Note that we have $\phi(x)\ne 0$ or $\phi(y)\ne 0$. 
Without loss of generality we can assume that $\phi(x)\ne 0$. 
It follows from  Lemma~\ref{lem:units} that $x-1$ has a left-inverse  in 
$\what{\Z[ G]_\phi}$. In particular this implies that 
$H_0(G;\what{\Z[ G]_\phi})=0$.

\begin{claim}
We have $H_1(G;\what{\Z[ G]_\phi})=0$ if and only if  $r_y$  has a left-inverse in $\what{\Z[ G]_\phi}$.
\end{claim}

We first suppose that $H_1(G;\what{\Z[ G]_\phi})=0$.
The row-vector $ ( -(y-1)(x-1)^{-1},\,\, 1 )$
lies in the kernel of $\partial_1$. Since $H_1(G;\what{\Z[ G]_\phi})=0$ it follows that there exists an $f\in \what{\Z[ G]_\phi}$ with 
\[ \partial_2 f=\bp fr_x , fr_y\ep=    \bp -(y-1)(x-1)^{-1},\,\, 1 \ep.\]
We thus showed that $r_y$ has a left-inverse in $\what{\Z[ G]_\phi}$.

Now we suppose that  $r_y$ has a left-inverse  in $\what{\Z[ G]_\phi}$. 
Let $(u, v) \in  \ker(\partial_1)$. This means that 
$0=\partial_1 ( u, v) =u(x-1)+v(y-1).$
We set $f:=vr_y^{-1}$. It then follows from (\ref{equ:foxzero})  that 
\[ \partial_2f=(f r_x, fr_y) =( -vr_y^{-1}\cdot r_y(y-1)(x-1)^{-1},\,\,v) =(-v(y-1)(x-1)^{-1},\,\,v) =(u,v).\]
We thus showed that  $H_1(G;\what{\Z[ G]_\phi})=0$.
This concludes the proof of the claim. 

The claim and the discussion preceding the claim, together with
Theorem~\ref{thm:sikorav} imply that $\phi$ represents an element in $\S(G)$ if and only if $r_y$ has a left-inverse in $\what{\Z[ G]_\phi}$.
Now we have the following claim.

\begin{claim}
The Fox derivative $r_y$ has a left-inverse  in $\what{\Z[ G]_\phi}$ if and only  if $\phi$ pairs maximally with a marked vertex of $\MM(r_y)$.
\end{claim}

We first suppose that $G=G_\pi$ is torsion-free.
It follows from \cite{Brj80,Brj84} that $G$ is locally indicable (see also \cite[Theorem~4.2.9]{CZ93} and \cite{Ho00}).
It follows from Lemma~\ref{lem:unitslocallyindicable} that  $r_y$ has a left-inverse  in $\what{\Z[ G]_\phi}$ if and only if $T_\phi(r_y)$ is a monomial.
But this statement in turn is equivalent to  $\phi$ pairing maximally with a marked vertex of $\MM(r_y)$.

Now we consider the case that $G$ has torsion elements.
We will show that $\MM(r_y)$ has no marked vertices and we will show that $r_y$ does not have a left-inverse in  $\what{\Z[ G]_\phi}$.

By  \cite[Theorem~IV.5.2]{LS77} the assumption that $G$ has torsion elements  implies that $r$ can be written as $r=s^m$ where $s$ is a cyclically reduced word that can not be written as a proper power  and with $m\geq 2$. An elementary calculation shows that 
\[ r_y=\frac{\partial}{\partial y}(s^m)=(1+s+\dots+s^{m-1})\frac{\partial s}{\partial y}.\]
Note that $s$ is torsion, in particular it represents the trivial element in $H_1(G;\R)$. Since $m\geq 2$ it  follows that for each vertex $v$ of $\PP(r_y)$ the $v$-component $(r_y)^v$ is the sum of at least $m$ elements in $G$. This in turn implies that no vertex of $\MM(r_y)$ is marked.

In order to proof the claim it remains to show that 
$r_y$ does not have a left-inverse in  $\what{\Z[ G]_\phi}$.
We write  $\ol{\pi}:=\ll x,y\,|\,s\rr$ and $\ol{G}:=G_{\ol{\pi}}$. 
By the above $s$ is not a proper power, which  implies by \cite[Theorem~IV.5.2]{LS77}  that $\ol{G}$ is locally indicable. It follows from $b_1(G)=2$ that $s$ is homologically trivial in $H_1(G;\Z)$ which implies that $H_1(G;\Z)\to H_1(\ol{G};\Z) $ is an isomorphism.
 By a slight abuse of language we denote the map $H_1(\ol{G};\Z)\cong H_1(G;\Z)\xrightarrow{\phi}\Z$ again by $\phi$. 
The projection  $G\to \ol{G}$ induces  ring homomorphisms $f\colon \Z[G]\to \Z[\ol{G}]$ and $f\colon \what{\Z[ G]_\phi}\to \what{\Z[ \ol{G}]_\phi}$.
Since $s$ is trivial in $\ol{G}$ we have 
\[ \textstyle f(r_y)=f\big((1+s+\dots+s^{m-1})(\frac{\partial s}{\partial y})\big)=f(1+s+\dots+s^{m-1})f(\textstyle\frac{\partial s}{\partial y})=mf(\frac{\partial s}{\partial y}).\]
In particular every coefficient of $f(r_y)\in \Z[\ol{G}]$ is divisible by $m\geq 2$. 
Since $\ol{G}$ is locally indicable it follows from Lemma~\ref{lem:unitslocallyindicable} that $f(r_y)$ does 
 does not have a left-inverse in  $\what{\Z[\ol{ G}]_\phi}$.
 But this implies that  $r_y$ does not have a left-inverse in  $\what{\Z[ G]_\phi}$.
This concludes the proof of the claim.

Now the theorem is an immediate consequence of the following claim.

\begin{claim}
Let $\psi\in \hom(G;\R)$ with $\psi(x)\ne 0$. Then $\psi$ pairs maximally with a marked vertex of $\MM(r_y)$ if and only if $\psi$ pairs maximally with a marked vertex of $\MM_\pi$.
\end{claim}

By Proposition~\ref{prop:foxpolytope} we have $\MM(r_y)=\MM_\pi+\MM(x-1)$.
We first suppose that $\psi$ pairs maximally with a marked vertex $v$ of $\MM_\pi\subset H_1(G;\R)=\R^2$. If $\psi(x)>0$, then $\psi$ pairs maximally with $v+(1,0)\in \MM(r_y)=\MM_\pi+\MM(x-1)$. In particular $v+(1,0)$ is a vertex of $\MM(r_y)$, and as the sum of two marked vertices it is also marked. If $\psi(x)<0$, then almost the same argument shows that $v$ itself is the desired marked vertex of $\MM(r_y)$.

Now suppose that $\psi$ pairs maximally with a marked vertex $w$ of $\MM(r_y)$.
A slight variation on the argument above shows the following:
if $\psi(x)>0$, then $v-(1,0)$  is the desired marked vertex of $\MM_\pi$,
and if $\psi(x)<0$, then $v$ is again the desired marked vertex.
\end{proof}

\section{Example}
\label{section:example}

We consider again the example which was already studied by Brown \cite[Section~4]{Brn87}.
Namely let
\[ \pi=\ll x,y\,|\, x^{-1}y^{-1}xy^2x^{-1}y^{-1}x^{2}y^{-1}x^{-1}yx^{-1}yxy^{-1}\rr.\]
A direct calculation  shows that
\[ \ba{rcl} r_x&=&-x^{-1}+x^{-1}y^{-1}-x^{-1}y^{-1}xy^2x^{-1}+x^{-1}y^{-1}xy^2x^{-1}y^{-1}\\
&&x^{-1}y^{-1}xy^2x^{-1}y^{-1}x-x^{-1}y^{-1}xy^2x^{-1}y^{-1}x^{2}y^{-1}x^{-1}\\
&&-x^{-1}y^{-1}xy^2x^{-1}y^{-1}x^{2}y^{-1}x^{-1}yx^{-1}+x^{-1}y^{-1}xy^2x^{-1}y^{-1}x^{2}y^{-1}x^{-1}yx^{-1}yx.\ea \]
By Corollary~\ref{cor:distinct} all these terms represent distinct elements in $G_\pi$. We sort these terms according to their abelianizations.
We see that
\[  \ba{rcl} r_x&=&\hspace{0.325cm}
x^0y^0\cdot x^{-1}y^{-1}xy^2x^{-1}y^{-1}x\\
&&+  x^{-1}y^{-1}\\
&&-x^0y^{-1}\cdot x^{-1}y^{-1}xy^2x^{-1}y^{-1}x^{2}y^{-1}x^{-1}y\\
&&+ x^{-1}y^0\cdot (-1+x^{-1}y^{-1}xy^2x^{-1}y^{-1}x-x^{-1}y^{-1}xy^2x^{-1}y^{-1}x^{2}y^{-1}x^{-1}y)\\
&&+x^{-1}y\cdot (-x^{-1}y^{-1}xy^2x^{-1}y^{-1}x+
x^{-1}y^{-1}xy^2x^{-1}y^{-1}x^{2}y^{-1}x^{-1}y).\ea \]
Therefore the polytope corresponding to $r_y$ is spanned by $(0,0), (-1,-1), (0,-1), (-1,0)$  and $(-1,1)$.
The  vertices of this polytope are $(0,0), (-1,-1), (0,-1)$ and $(-1,1)$, among which $(0,0), (-1,-1)$ and $(0,-1)$ are marked and the vertex $(-1,1)$ is unmarked.

In Figure \ref{fig:brown} we show how to obtain $\MM(r_x)$ and $\MM_\pi$
and we  indicate the set of all $\phi$'s which pair maximally with a marked vertex of $\MM_\pi$.

\begin{figure}[h]
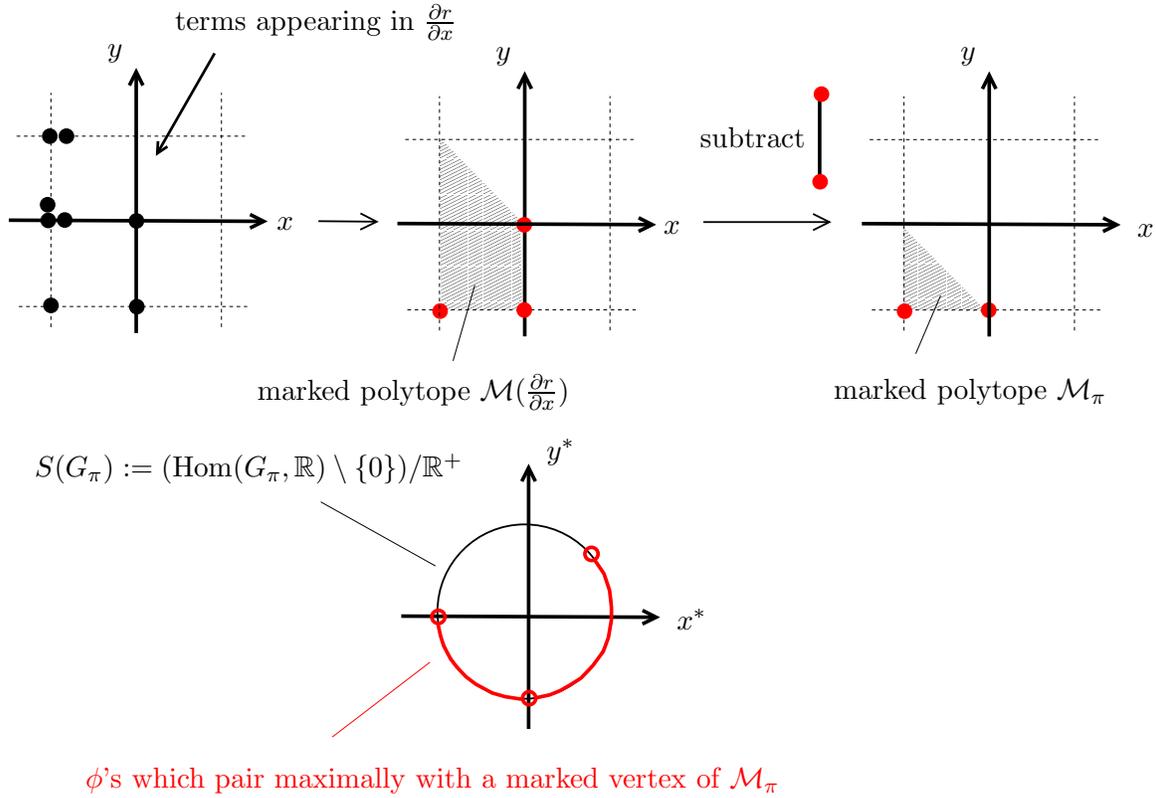\label{fig:brown}
\caption{The marked polytope of Brown's example.}
\end{figure}

\section{Proof of Theorem~\ref{mainthm2}}
\label{sec:proof of second theorem}

We denote by $\GG$ the class of all groups that are torsion-free and elementary amenable. Note that $\GG$ is closed under taking subgroups and finite direct products. We say that a group $G$ is \emph{residually $\GG$} if given any non-trivial $g\in G$ there exists a homomorphism $\a\co G\to \G$ with $\G\in \GG$ such that $\a(g)$ is non-trivial.

For the reader's convenience we recall the statement of  Theorem~\ref{mainthm2}.\\

\noindent \textbf{Theorem~\ref{mainthm2}}\emph{
Let  $G$ be a group which admits a nice $(2,1)$--presentation $\pi=\ll x,y\,|\, r\rr$.
If $G$ is residually $\GG$, then the polytope $\MM_\pi\subset H_1(G;\R)$ is an invariant of the group $G$ (up to translation).}\\

The following lemma gives a criterion for when the hypothesis in 
Theorem~\ref{thm:polytope} is satisfied. 

\begin{lemma} \label{lem:zbyfree}
Let  $G$ be a group which admits a $(2,1)$--presentation $\pi$.
If  there exists a $\phi\in S(G)$ such that both $\phi$ and $-\phi$ lie in $\S(G)$, then $G$ is residually  a torsion-free solvable group, in particular $G$ is residually $\GG$.
\end{lemma}

The criterion from Lemma \ref{lem:zbyfree} applies to the example provided  in Section~\ref{section:example}. Indeed, the homomorphism $\phi\colon G_\pi\to \Z$ defined by $\phi(x)=2$ and $\phi(y)=1$ has the property that $\phi$ pairs maximally with the marked vertex $(-1,-1)$ and  $-\phi$  pairs maximally with the marked vertex $(0,-1)$. It follows from Theorem~\ref{mainthm} that both $\phi$ and $-\phi$ represent elements  in $\S(G)$.

\begin{proof}
Let  $G$ be a group which admits a $(2,1)$--presentation $\pi$ and suppose $\phi\in \hom(G,\R)$ is a homomorphism such that both $\phi$ and $-\phi$ represent elements in $\S(G)$.
It follows from the openness of $\S(G)$ that without loss of generality we can assume that $\phi$ takes values in $\Z$, i.e.\ that  $\phi\in \hom(G,\Z)$. Recall, see Section~\ref{section:bns}, that the existence of such a $\phi$  implies that $\ker(\phi)$ is a finitely generated group.

Since $G$  admits a $(2,1)$--presentation 
it  follows from \cite[p.~487]{Brn87}, see also \cite[Corollary~B]{Bi07}, that $\ker(\phi)$ is a free group. This implies that  $G$  is isomorphic to a semidirect product $\Z\ltimes F$ where $F$ is a  free group.
We then consider the filtration
\[ G\supset F\supset F^{(1)}\supset F^{(2)}\supset \dots \]
where $F^{(n)}$ denotes the $n$-th group in the derived series of $F$. 
Each successive quotient is a free abelian group. Also note that each $F^{(n)}$ is characteristic in $F$ and it is thus a normal subgroup of $G$. It follows that each quotient
$G/F^{(i)}$ is a torsion-free solvable group. Now the lemma follows from the well-known fact that $\cap F^{(i)}$ is trivial.
\end{proof}

\subsection{The Ore localization of group rings}\label{section:ore}

Let $\G$ be a group which lies in $\GG$.
It follows from \cite[Theorem~1.4]{KLM88} that the group ring $\Z[\G]$ is a domain, i.e.\ it has no non-zero element which is a left or right zero-divisor.
Since  $\G$ is in particular amenable  it follows from \cite[Corollary~6.3]{DLMSY03} that $\Z[\G]$ satisfies the Ore condition. This means that for any
two non-zero elements $x,y\in \Z[\G]$ there exist non-zero elements $p,q\in \Z[\G]$ such that $xp=yq$.

Now we denote by $\K(\G)$ the set of equivalence classes of pairs $(p,q)$ where $p\in \Z[\G]$ and $q\in \Z[\G]\sm \{0\}$.
Here we say that two such pairs $(p,q)$ and $(p',q')$ are equivalent if there exist non-zero $x,y\in \Z[\G]$ with $xp=yp'$ and $xq=yq'$.  As usual we denote
such an equivalence class by $pq^{-1}$. Since $\Z[\G]$ is a domain it follows that the canonical map $\Z[\G]\to \K(\G)$ is injective.
By  \cite[Section~4.4]{Pa77}
we can extend the ring structure on $\Z[\G]$ to a ring structure on $\K(\G)$, and with this ring structure, $\K(\G)$ is actually a skew field that contains $\Z[\G]$ as a subring.

\begin{remark}
The Zero-Divisor Conjecture states that for any torsion-free group the group ring $\Z[\G]$ is a domain. If this conjecture holds for all torsion-free amenable groups, then throughout the paper we could work with the class of torsion-free amenable groups instead of torsion-free elementary amenable groups.
\end{remark}

\subsection{Non-commutative Reidemeister torsion of presentations}

Let $X$ be a finite CW-complex with $G=\pi_1(X)$. We denote by $\wti{X}$ the universal cover of $X$.
Let $\varphi\co G\to \G$ be a homomorphism to a group $\G\in \GG$.
The homomorphism $\varphi$ equips  $\Z[\G]$ and $\K(\G)$ with the structure of a right $\Z[G]$-module.
Following the discussion in Section~\ref{section:twihom} we can thus consider
 the chain complexes $C_*^\varphi(X;\Z[\G]):=\Z[\G] \otimes_{\Z[G]} C_*(\wti{X})$
and  $C_*^\varphi(X;\K(\G)):=\K(\G) \otimes_{\Z[G]}C_*(\wti{X})$.
If $C_*^\varphi(X;\K(\G))$ is not acyclic, then we define the corresponding Reidemeister torsion $\tau(X,\varphi)$ to be zero.
Otherwise we pick an ordering of the cells of $X$ and  for  each cell in $X$ we pick a lift to $\wti{X}$. This  turns $C_*^\varphi(X;\K(\G))$ into a chain complex of based $\K(\G)$-left modules. We then define 
\[ \tau(X,\varphi) \in K_1(\K(\G)).\]
to be the Reidemeister torsion of the based chain complex $C_*^\varphi(X;\K(\G))$.
(Here, given a ring $R$ the first $K$-group  $K_1(R)$ is defined as the abelianization of $\underset{\longrightarrow}{\lim} \gl(n,R)$.)
Now we write $\K(\G)^\times =\K(\G)\sm \{0\}$
and we denote by $\K(\G)^\times_{\op{ab}}$ the abelianization of the multiplicative group $\K(\G)^\times$.
 The Dieudonn\'e determinant, see \cite{Ro94}, 
gives rise to an isomorphism $K_1(\K(\G))\to \K(\G)^\times_{\op{ab}}$ which we will use to identify these two groups.
The invariant  $\tau(X,\varphi)\in \K(\G)^\times$ is well-defined up to multiplication by an element of the form $\pm g$ with $g\in \G$. 
Furthermore this invariant  only depends on the homeomorphism type of $X$ and the choice of $\varphi$.
We refer to \cite{Tu01,Fr07,FH07} for details and more precise references.

\begin{example}
Given an oriented $m$-component link $L\subset S^3$ we denote by $X_L=S^3\sm \nu L$ the  exterior of $L$, i.e.\ the complement of an open tubular neighborhood $\nu L$ of $L$. We equip $X_L$ with a CW-structure. We denote by $T$ the multiplicative free abelian group generated by $t_1,\dots,t_m$. 
Furthermore we  denote by $\varphi\co \pi_1(X_L)\to T$ the canonical epimorphism given by sending the $i$-th oriented meridian to $t_i$.
Finally we denote by $\Delta_L(t_1,\dots,t_m)$ the multivariable Alexander polynomial of $L$. 
It follows from \cite{Tu01} that 
\[ \tau(X_L,\varphi)=\left\{ \ba{ll} \frac{\Delta_L(t_1)}{t_1-1},&\mbox{if $L$ has one component}, \\ \Delta_L(t_1,\dots,t_m),&\mbox{if $L$ has more than one component.}\ea \right.\]
Thus the invariant $\tau(X_L,\varphi)$ for admissible homomorphisms to non-abelian groups can  be viewed as a non-commutative generalization of the 
Alexander polynomial of a link. The first such invariants were introduced in \cite{Coc04} for knots, in \cite{Ha05} for general 3--manifolds
and in \cite{LM06,LM08} for plane algebraic curves. 
\end{example}

In the following, given a presentation $\pi$ and a homomorphism $\varphi\co G_\pi\to \G$ to a group in $\GG$ we  write
\[ \tau(\pi,\varphi)=\tau(X_\pi,\varphi)\]
where $X_\pi$ is the 2-complex corresponding to the presentation $\pi$.

\subsection{The polytope group}
Let $V$ be a vector space. 
We denote by $\mathfrak{P}(V)$ the set of all translation-equivalence classes of polytopes in $V$. With the Minkowski sum this becomes an abelian monoid, where the identity element $0$ is given by the polytopes consisting of a single point. 
It is straightforward to show, see e.g.\ \cite[Lemma~3.1.8]{Sc93},  that $\mfp(V)$ has the   cancellation property, i.e.\ for $\PP,\QQ,\RR\in \mfp(V)$ with $\PP+\QQ=\PP+\RR$ we have $\QQ=\RR$. 

We denote by $\mfg(V)$ the set of all equivalence classes of pairs $(\PP,\QQ)\in \mfp(V)^2$ where we say that $(\PP,\QQ)\sim (\PP',\QQ')$ if $\PP+\QQ'=\PP'+\QQ$.
Note that $\mfg(V)$ is an abelian group, and since $\mfp(V)$ has the cancellation property it follows that the map 
\[ \ba{rcl} \mfp(V)&\to &\mfg(V) \\
\PP&\mapsto &(\PP,0)\ea\]
is  a monomorphism. We will use this monomorphism to identify $\mfp(V)$ with its image in $\mfg(V)$. As usual, given $\PP$ and $\QQ\in \mfp(V)$ we write $\PP-\QQ=(\PP,\QQ)$. With our conventions this is consistent with the definition of $\PP-\QQ$ given in Section~\ref{section:minkowski}.

Let $\G$ be a group in $\GG$. We write $V=H_1(\G;\R)$. 
In Section~\ref{section:ore} we saw that $\Z[\G]$ is a domain. It follows  from Lemma~\ref{lem:productadd} that 
\[ \ba{rcl} \PP\co \Z[\G]\sm \{0\} &\to&\mfp(V) \\ f&\mapsto &\PP(f)\ea \]
is a homomorphism of monoids. Since $\mfg(V)$ is commutative this extends to a group homomorphism
\[ \PP\co \K(\G)^\times_{\op{ab}} \to \mfg(V)\]
which we also denote by $\PP$. 

\subsection{The invariant $\TT(\pi)$ for a  $(2,1)$--presentation $\pi$}
\label{section:tg}
An \emph{admissible homomorphism for a group $G$} is 
an epimorphism  $\varphi\co G\to \G$  to a group $\G\in \GG$ such that the 
projection map $G\to H_1(G;\Z)/\mbox{torsion}$ factors through $\varphi$.
Note that $\varphi$ induces  an isomorphism $H_1(G;\R)\cong H_1(\G;\R)$.
Throughout this paper, given an admissible homomorphism $\varphi\co G\to \G$ we will use $\varphi$ to identify $H_1(G;\R)$ with $H_1(\G;\R)$.

If $\pi=\ll x,y|r\rr$ is a nice $(2,1)$--presentation and if $\varphi\co G_\pi\to \G$ is an admissible homomorphism, then it follows in particular that $\varphi(x)$ and $\varphi(y)$ are non-trivial, since they are already non-trivial
in $H_1(\G;\Z)\cong H_1(G_\pi;\Z)\cong \Z^2$.

We will repeatedly make use of the following observation.

\begin{lemma}\label{lem:capadmissible}
If   $\varphi_1$ and $\varphi_2$ are two admissible homomorphisms for $G$, then the projection map
$G\to G/\ker(\varphi_1)\cap \ker(\varphi_2)$ is also admissible.
\end{lemma}

We also need the following lemma.

\begin{lemma}\label{lem:computetau} 
Let $\pi=\ll x,y\,|\,r\rr$ be a nice $(2,1)$--presentation. Let  $\varphi\co G_\pi\to \G$ be an
admissible homomorphism. Then the following are equivalent:
\bn
\item  $\varphi(r_x)\ne 0$,
\item $\varphi(r_y)\ne 0$,
\item $\tau(X,\varphi)\ne 0$.
\en
Furthermore, if any of the three equivalent statement holds, then 
\[ \PP(\tau(X_\pi,\varphi))=\PP(\varphi(r_x))-\PP(\varphi(y-1))=\PP(\varphi(r_y))-\PP(\varphi(x-1))\]
where the equality holds in $\mfg(H_1(\G;\R))=\mfg(H_1(X_\pi;\R))$.
\end{lemma}

\begin{proof}
As remarked above, $\varphi(x)$ and $\varphi(y)$ are non-trivial. This implies
that  $\varphi(x-1)$ and $\varphi(y-1)$
are invertible in $\K(\G)$. Now the  lemma is an immediate consequence of the definitions and Theorem~2.1 of  \cite{Fr07} which says in this context that
\[ \tau(X,\varphi)=\varphi(r_x)\varphi(y-1)^{-1}=\varphi(r_y)\varphi(x-1)^{-1}.\]
\end{proof}

Let $V$ be a vector space.
Given  $(\PP,\QQ)$ and $(\PP',\QQ')$ in $\mfg(V)$ we  write $(\PP,\QQ)\leq (\PP',\QQ')$ if there exists a $v\in V$ such that 
$v+\PP+\QQ'\subset \PP'+\QQ$. Note that this descends to a  partial ordering  on $\mfg(V)$.

Now we have the following lemma, which is a straightforward consequence of the definitions, of Proposition~\ref{prop:foxpolytope} and of Lemma~\ref{lem:computetau}.
We leave the details to the reader.

\begin{lemma}\label{lem:comparetau} 
Let $\pi=\ll x,y\,|\,r\rr$ be a nice $(2,1)$--presentation. Let  $\varphi\co G_\pi\to \G$ be an
admissible homomorphism. Then
\[ \PP(\tau(X_\pi,\varphi))\leq \PP_\pi\]
as polytopes in $H_1(\G;\R)=H_1(G_\pi;\R)$.
If $\psi\co G_\pi\to \G'$ is an admissible homomorphism which factors through $\varphi$,
then
\[ \PP(\tau(X_\pi,\psi))\leq \PP(\tau(X_\pi,\varphi))\]
as polytopes in $H_1(\G;\R)=H_1(\G';\R)=H_1(G_\pi;\R)$.
\end{lemma}
 
We have the following corollary.

\begin{corollary}\label{cor:existmaxtau}
Let $\pi=\ll x,y\,|\,r\rr$ be a $(2,1)$--presentation. There exists an admissible  $\varphi$ such that for any 
other admissible homomorphism $\psi$ we have 
\[ \PP(\tau(X_\pi,\psi))\leq \PP(\tau(X_\pi,\varphi)).\]
\end{corollary}

\begin{proof}
Given a polytope $\PP\subset \R^2$ we denote by $\ell(\PP):= \# \left(\PP\cap \Z^2\right)$ the number of lattice points. In the following we follow the usual convention and we identify $H_1(G_\pi;\Z)$ with $\Z^2$.
The following two statements are an immediate consequence of Lemma~\ref{lem:comparetau} 
\bn
\item If  $\varphi\co G_\pi\to \G$ is an
admissible homomorphism, then
\[ \ell(\PP(\tau(X_\pi,\varphi)))\leq \ell(\PP_\pi).\]
\item If $\psi\co G_\pi\to \G'$ is an admissible homomorphism which factors through $\varphi$,
then
\[ \ell(\PP(\tau(X_\pi,\psi)))\leq \ell(\PP(\tau(X_\pi,\varphi))).\]
\en
We pick an  admissible homomorphism $\varphi$ such that 
$ \ell(\PP(\tau(X_\pi,\varphi)))$ is maximal among all admissible homomorphism.
This definition makes sense since the values for $ \ell(\PP(\tau(X_\pi,\varphi)))$ are bounded by the finite number $\ell(\PP_\pi)$ and since there exists always at least one admissible homomorphism, namely the abelianization homomorphism $G_\pi\to H_1(G_\pi;\Z)\cong \Z^2$. 

We claim that $\varphi$ has the desired property. So suppose that 
$\psi$ is another admissible homomorphism. We want to show that 
\[ \PP(\tau(X_\pi,\psi))\leq \PP(\tau(X_\pi,\varphi)).\]
We consider the  homomorphism $\phi\colon G_\pi\to \ker(\varphi)\cap \ker(\psi)$ which is admissible by  Lemma~\ref{lem:capadmissible}. Since  $\varphi$ factors through $\phi$ it follows from Lemma~\ref{lem:comparetau} that 
\[ \PP(\tau(X_\pi,\varphi))\leq \PP(\tau(X_\pi,\phi)).\]
On the other hand, by the choice of $\varphi$ we have 
\[ \ell(\PP(\tau(X_\pi,\phi)))\leq \ell(\PP(\tau(X_\pi,\varphi))).\]
Since $\PP(\tau(X_\pi,\varphi))$ and $\PP(\tau(X_\pi,\phi))$ both have vertices in $H_1(G_\pi;\Z)=\Z^2$ it follows that 
\[ \PP(\tau(X_\pi,\phi))=\PP(\tau(X_\pi,\varphi)).\]
The desired inclusion now follows from  Lemma~\ref{lem:comparetau} which also says that 
\[ \PP(\tau(X_\pi,\psi))\leq \PP(\tau(X_\pi,\phi)).\]
\end{proof}

Therefore, given a $(2,1)$--presentation $\pi$ it makes sense to define
\[  \TT(\pi)=\max\{ \PP(\tau(\pi,\varphi))\,|\, \varphi\mbox{ admissible homorphism}\} \in \mfp(H_1(G_\pi;\R)).\]

\subsection{Proof of Theorem~\ref{thm:polytope}}
Now we are ready to prove Theorem~\ref{mainthm2}. We start out with the following proposition.

\begin{proposition}\label{prop:invtofg}
Let $\pi$ and $\pi'$ be $(2,1)$--presentations.
If $f\co G_\pi\to G_{\pi'}$ is an isomorphism and if $G_\pi\cong G_{\pi'}$ is torsion-free, then
\[ f_*(\TT(\pi))=\TT(\pi')\in \mfp(H_1(G_{\pi'};\R)).\]
\end{proposition}

\begin{proof}
By \cite[Proposition~11.1]{LS77} the 2-complexes $X_\pi$ and $X_{\pi'}$  corresponding to the  $(2,1)$--presentations $\pi$ and $\pi'$ are aspherical.
It follows that $f$  is induced by a homotopy equivalence $f\co X_\pi\to X_{\pi'}$.
Since $\pi$ and $\pi'$ are presentations of torsion-free one-relator groups 
it follows from work of Waldhausen \cite[p.~249~and~p.~250]{Wa78} that the Whitehead group of $G_\pi\cong G_{\pi'}$ is trivial, which implies 
that $f$   induces in fact a simple homotopy equivalence $f\co X_\pi\to X_{\pi'}$. 

For any admissible homomorphism $\varphi\co G_{\pi'}\to \G$ the homomorphism $\varphi\circ f_*$ is an admissible homomorphism for $G_\pi$. Evidently all admissible homomorphism for $G_\pi$ are of that form. Since $f$ is a simple homotopy we  have
\[ f_*(\tau(X_{\pi},\varphi \circ f))=\tau(X_{\pi'},\varphi).\]
Now the proposition is an immediate consequence of these observations and the definitions.
\end{proof}

We also have following proposition.

\begin{proposition}\label{prop:dpphiequalsx}
Let   $\pi=\ll x,y\,|\, r\rr$ be  a nice $(2,1)$--presentation. 
If $G_\pi$ is residually $\GG$, then
\[ \TT(\pi)=\PP_\pi\in \mfp(H_1(G_\pi;\R)).\]
\end{proposition}

\begin{proof}
Let   $\pi=\ll x,y\,|\, r\rr$ be  a nice $(2,1)$--presentation.
Denote by $\psi\co G_\pi\to H_1(G_\pi;\Z)/\mbox{torsion}$ the canonical projection map,
and assume $G_\pi$ is residually $\GG$. Recall that given any non-trivial $g\in G_\pi$
there exists a  homomorphism $\varphi\co G_\pi\to \G$ to a group in $\GG$ such that 
$\varphi(g)$ is non-trivial. Note that
\[  G_\pi\to G_\pi\,\,/\,\,(\ker(\varphi)\cap \ker(\psi))\]
is an admissible homomorphism to a group in $\GG$ such that the image of $g$ is non-trivial.

It follows from  Lemma~\ref{lem:capadmissible} that given any finite collection of elements $\{g_i\} \subset G_\pi,$ there exists
an admissible homomorphism $\varphi\co G_\pi\to \G$  such that  the images $\varphi(g_i)$ are pairwise distinct. We apply this to the set of non-trivial elements appearing in $r_x,$ and 
as before we identify $H_1(\G;\R)$ with $H_1(\pi;\R)$. We write $V=H_1(\G;\R)=H_1(G_\pi;\R)$.

Since the $\varphi(g_i)$ are pairwise distinct it follows immediately from the definitions that 
\[ \PP(\varphi(r_x))=\PP(r_x)\subset V.\]
Also, note that  $y$ and $\varphi(y)$ represent the same  non-trivial element in $V$. It thus follows that
\[ \PP(\varphi(y-1))=\PP(y-1).\]
Combining these two equalities with Proposition~\ref{prop:foxpolytope} we obtain that 
\[ \PP(\tau(X_\pi,\varphi))=\PP_\pi.\]
If we combine this equality with Lemma~\ref{lem:comparetau} and Corollary~\ref{cor:existmaxtau}
 we see that 
\[ \PP_\pi=\PP(\tau(X_\pi,\varphi)) \subset \TT(\pi)\subset \PP_\pi.\]
It thus follows that $ \TT(\pi)= \PP_\pi$.
\end{proof}

\begin{proof}[Proof of Theorem~\ref{thm:polytope}]
Let  $G$ be a group which admits a nice $(2,1)$--presentation 
and which has the property that $G$ is residually $\GG$. Our assumption implies in particular that the group  $G$ is residually a torsion-free group, which in turn implies that $G$ itself is torsion-free.

 Let $\pi'$  be another nice $(2,1)$--presentation for $G$. We write $V=H_1(G;\R)$.
It follows from Propositions \ref{prop:invtofg}  and \ref{prop:dpphiequalsx}
that $\PP_\pi=\PP_{\pi'}\in \mfg(V)$. 
Since the Bieri--Neumann--Strebel invariant is an invariant of the group $G$ it  follows from Theorem~\ref{mainthm} that $\PP_\pi$ and $\PP_{\pi'}$ 
have the same marked vertices, i.e.\ we have $\MM_\pi=\MM_{\pi'}$.
\end{proof}

\section{Proof of Theorem~\ref{mainthm3}}
\label{sec:proof of third theorem}

\subsection{Thickness} 

We recall that  given a polytope $\PP$ in a vector space $V$ and a homomorphism $\phi\co V\to \R$ we define the \emph{thickness of $\PP$ with respect to $\phi$}
as 
\[ \th(\PP,\phi)=\max\{ \phi(p)-\phi(q)\,|\,p,q\in \PP\}.\]
Furthermore, we refer to
\[ \PP^{\sym}:=\{ \tfrac{1}{2}(p-q)\,|\, p,q\in \PP\}\]
as the \emph{symmetrization of $\PP$}. 

For future reference we record the following lemma. 
We will only use the first part which is a straightforward consequence of the definitions. We include the second part to facilitate a discussion later on.
We leave the elementary proof to the reader.

\begin{lemma}\label{lem:addth}
Let $\PP$ and $\QQ$ be polytopes in a vector space $V$. Then the following hold:
\bn
\item
If  $\phi\co V\to \R$ is a homomorphism, then we have
\[ \th(\PP+\QQ,\phi)=\th(\PP,\phi)+\th(\QQ,\phi).\]
\item We have 
\[ \th(\PP,\phi)=\th(\QQ,\phi)\mbox{ for all }\phi\in \hom(V,\R)\]
if and only if $\PP^{\sym}=\QQ^{\sym}$.
\en
\end{lemma}

\subsection{Splittings of groups}\label{section:splittings} 
Let $G$ be a finitely presented group and let $\phi\co G\to \Z$ be an epimorphism. Let $B$ be a finitely generated group. A 
\emph{splitting of $(G,\phi)$ over  $B$} is an isomorphism
\[ f\co G\xrightarrow{\cong} \ll A,t\,|\, \mu(B)=tBt^{-1}\rr \]
such that the following hold:
\bn
\item $A$ is finitely generated,
\item  $B$ is a subgroup of $A$ and $\mu\co B\to A$ is a monomorphism,
\item  $(\phi \circ f^{-1})(x)=0$ for $a\in A$  and $(\phi\circ f^{-1})(t)=1$.
\en
It is well-known, see e.g.\ \cite{BS78} or \cite[Theorem~B*]{Str84}, that any such pair
$(G,\phi)$ admits a splitting  over a finitely generated group.
We define the \emph{splitting complexity of $(G,\phi)$} as 
\[ c(G,\phi)=\min\{ \rank(B)\,|\, (G,\phi) \mbox{ splits over }B\},\]
where $\rank(B)$ is defined as the minimal number of generators of $B$.

In the following we will also consider the \emph{free splitting complexity $c_f(G,\phi)$}. If $(G,\phi)$ does not split over a free group, then we define $c_f(G,\phi)=\infty$, otherwise we define the free complexity to be
\[ c_f(G,\phi)=\min\{ \rank(F)\,|\, (G,\phi) \mbox{ splits over a free group } F\}. \]
By definition we have $c(G,\phi)\leq c_f(G,\phi)$. 

\begin{example}
Let $K$ be a knot in $S^3$. It follows easily from the definitions and the  Seifert-van Kampen theorem that 
\[ c(G,\phi)\leq c_f(G,\phi)\leq \mbox{$2\cdot \mbox{genus}(K)$}\]
where $\genus(K)$ denotes the minimal genus of a Seifert surface.
In \cite{FSW13} it was shown that the above inequalities are in fact  equalities.
\end{example}

The following theorem is  a slightly stronger version of Theorem~\ref{mainthm3}.

\begin{theorem}\label{thm:splittings}
Let  $G$ be a group which admits a nice $(2,1)$--presentation.
If $G$ is residually $\GG$, then for any epimorphism $\phi\co G\to \Z$ we have
\[ c(G,\phi)-1=c_f(G,\phi)-1=\th(\PP_\pi,\phi).\]
\end{theorem}

The theorem is an immediate consequence of 
Propositions \ref{prop:splittings1} and \ref{prop:splittings2} which we will prove in the next two sections.

\subsection{Upper bounds on the complexity of splittings}
In this section we give an upper bound on $c_f(G,\phi)$.
This result does not require any extra assumptions on $G$.
We are very grateful to Nathan Dunfield for telling us about this proposition.

\begin{proposition}\label{prop:splittings1}
Let  $G$ be a group which admits a nice $(2,1)$--presentation. Then for any epimorphism $\phi\co G\to \Z$ we have
\[ c_f(G,\phi)-1\leq \th(\PP_\pi,\phi).\]
\end{proposition}

We first prove the following lemma.

\begin{lemma}\label{lem:splittings1}
Let  $G$ be a group which admits a nice $(2,1)$--presentation. Let  $\phi\co G\to \Z$ be an epimorphism.
Then there exists a 
nice $(2,1)$--presentation $\ll x,y\,|\,r\rr$ for $G$ with $\phi(x)=0$, $\phi(y)=1$, and that gives rise to the same polytope as $\pi$.
\end{lemma}

\begin{proof}
Let $\pi=\ll x,y\,|\,r\rr$ be a nice $(2,1)$--presentation for $G$. We will prove the lemma by induction on
$|\phi(x)|+|\phi(y)|$. Without loss of generality we can assume that $|\phi(x)|\leq |\phi(y)|$. 

If $|\phi(x)|+|\phi(y)|=1$, then we are done.
So  suppose that $|\phi(x)|+|\phi(y)|>1$. 
Note that the assumption that $\phi$ is surjective means that the ideal generated by $\phi(x)$ and $\phi(y)$ is all of $\Z$.
The assumption that  $|\phi(x)|+|\phi(y)|>1$ thus implies that $\phi(x)\ne 0$.

Now we put $c=yx^\eps$ where $\eps=1$ if $\phi(x)$ and $\phi(y)$ have opposite signs and where $\eps=-1$  if $\phi(x)$ and $\phi(y)$ have the same signs.
Note that $|\phi(c)|<|\phi(y)|$. 

Now we replace every occurrence of $y$ in $r$ by $cx^{-\eps}$ and we then reduce and cyclically reduce the resulting word in $x$ and $c$.
We denote the resulting word by $s$. 
Now it is straightforward to see that as polytopes in $H_1(G;\R)$ we have $\PP(s_c)=\PP(r_y)$ and of course $\PP(x-1)=\PP(x-1)$.
We leave the details to the reader. 
It then follows that $\pi$ and $\ll a,c\,|\,s\rr$ give rise to the same polytope in $H_1(G;\R)$.
\end{proof}

It is clear that Lemma~\ref{lem:splittings1} together with the following lemma
proves Proposition~\ref{prop:splittings1}.

\begin{lemma}\label{lem:splittings2}
Let  $G$ be a group and let  $\phi\co G\to \Z$ be an epimorphism.
Suppose $G$ admits a nice $(2,1)$--presentation   $\pi=\ll a,t\,|\,r\rr$ such that   $\phi(a)=0$, $\phi(t)=1$. Then
\[ c_f(G,\phi)-1\leq \th(\PP_\pi,\phi).\]
\end{lemma}

\begin{proof}
After a cyclic permutation of the letters in $r$ we can and will assume
that $r$ is of the form $r=t^{m_1}x^{n_1}\cdot t^{m_2}x^{n_2}\cdot \cdot \dots \cdot t^{m_k}x^{n_k}$ where all the $m_i$ and $n_i$ are non-zero.
Given $j\in \{1,\dots,k\}$ we write $M_j=m_1+\dots+m_j$. We consider
\[ d:=\max\{ M_1,\dots,M_k\} \mbox{ and } D:=\min\{M_1,\dots,M_k\}.\]
Now we have the following claim.

\begin{claim}
The pair $(G,\phi)$ splits over a free group of rank $D-d$.
\end{claim}

Note that 
\[ r=\left(t^{M_1}xt^{-M_1}\right)^{n_1}\left(t^{M_2}xt^{-M_2}\right)^{n_2}\cdot \dots \cdot \left(t^{M_k}xt^{-M_k}\right)^{n_k}.\]
It thus follows from Tietze transformations that the  assignment $x_i\mapsto t^iat^{-i}$  induces an isomorphism
\[ \ll x_d,\dots,x_D,t| x_{M_1}^{n_1}\cdot \dots\cdot x_{M_k}^{n_k},x_{i+1}=tx_it^{-1}\mbox{ for }i=d,\dots,D-1\rr\xrightarrow{\cong} \ll a,t|r\rr.\]
We  write $A=\ll x_d,\dots,x_D|x_{M_1}^{n_1}\cdot \dots\cdot x_{M_k}^{n_k}\rr$.
It follows from the Freiheitssatz, see \cite[Section~II.5]{LS77}, that $x_d,\dots,x_{D-1}$ and $x_{d+1},\dots,x_D$ each generate a free subgroup of $A$.
Now we write $B=\ll x_d,\dots,x_{D-1}\rr$ and we denote by $\varphi\co B\to A$ the injective map which is given by $\varphi(x_i)=x_{i+1}$ for $i=d,\dots,D-1$.
Note that 
\[ \ll x_d,\dots,x_D,t| x_{M_1}^{n_1}\cdot \dots\cdot x_{M_k}^{n_k},x_{i+1}=tx_it^{-1}\mbox{ for }i=d,\dots,D-1\rr=\ll A,t|\varphi(B)=tBt^{-1}\rr.\]
We thus showed that  the pair $(G,\phi)$ splits over the free group $B=\ll x_d,\dots,x_{D-1}\rr$ of rank $D-d$.
This concludes the proof of the claim.

Now the lemma follows from the following claim.

\begin{claim}
\[ \th(\PP_\pi,\phi)=D-d-1.\]
\end{claim}

We note that the Fox derivative $r_t$ is given by
\[ r_t=\sum_{i=1}^k \prod_{j=1}^{i-1} t^{m_j}x^{n_j}\eps_i(1+t+\dots+t^{|m_{i}|})\]
where $\eps_i=1$ if $m_i>0$ and $\eps_i=-t^{-1}$ if $m_i<0$. 
It follows easily that $\th(\PP(r_t),\phi)=D-d-1$. Evidently we have
$\th(\PP(x-1),\phi)=0$. It  follows from Proposition~\ref{prop:foxpolytope} and the additivity of thickness, 
see Lemma~\ref{lem:addth} (1), that
\[ \th(\PP_\pi,\phi)=\th(\PP(r_t),\phi)-\th(\PP(x-1),\phi)=D-d-1.\]
This concludes the proof of the claim and thus of the lemma.
\end{proof} 

\subsection{Lower bounds on the complexity of splittings}
In this section we will prove the following proposition which gives us lower bounds on $c(G,\phi)$.

\begin{proposition}\label{prop:splittings2}
Let  $G$ be a group which admits a nice $(2,1)$--presentation.
Suppose that  $G$ is residually  $\GG$. Then for any epimorphism $\phi\co G\to \Z$ we have
\[ c(G,\phi)\geq \th(\PP_\pi,\phi) + 1.\]
\end{proposition}

This proposition is in fact a consequence of 
Proposition~\ref{prop:dpphiequalsx} and the following proposition.

\begin{proposition}\label{prop:splittings-lower-bound}
Let  $G$ be a group which admits a nice $(2,1)$--presentation.
Let $\varphi\co G\to \G$ be an admissible homomorphism
such that $\tau(X_\pi,\varphi)\ne 0$ and such that $\varphi(x)$ and $\varphi(y)$ are non-trivial. Then for any epimorphism $\phi\co G\to \Z$ we have
\[ c(G,\phi)\geq \th(\PP(\tau(X_\pi,\varphi)),\phi).\]
\end{proposition}

\begin{remark}
This proposition  is related to  \cite[Theorem~8.5]{FSW13} where we gave lower bounds on the  splitting complexity  in terms of twisted Reidemeister torsion. The proposition is also related to the lower bounds on the knot genus and Thurston norm  which were obtained by Cochran and Harvey \cite{Coc04,Ha05} in terms of degrees of higher-order Alexander polynomials.
\end{remark}

In the proof of Proposition~\ref{prop:splittings-lower-bound} we will need several results from \cite{Coc04,Ha05,Fr07}.
In order to state the results we need a few definitions.
Let $\G$ be a group and let $\phi\co \G\to \Z$ be an epimorphism. We write $\G'=\ker(\phi\co \G\to \Z)$. 
Given a $\Z[\G]$-left module $M$ we  define
\[ \dim_{\phi}(M):=\dim_{\K(\G')}\left(\K(\G') \otimes_{\Z[\G']} M\right).\]

\begin{example}
Let $\G=\ll t\rr$ and let $\phi=\id$.
As usual we identify the group ring of $\G=\Z$ with $\zt$. Let $A(t)$ be a $k\times k$-matrix over $\zt$ with $\det(A(t))\ne 0$. 
We then have $\G'=\{0\}$ and 
\[ \ba{rcl} \dim_{\phi}\left(\zt^k/A(t)\zt^k\right)&=&\dim_{\Q}\left(\Q \otimes_{\Z}\zt^k/\zt^k A(t)\right)\\[2mm]
&=&\dim_{\Q}\left(\qt^k/\qt^k A(t)\right)=\deg(\det(A(t)).\ea\]
Here and throughout the paper recall that given a ring $R$ we view elements in $R^k$ as row-vectors and matrices act on $R^k$ by right multiplication.
\end{example}

We will  need the following lemma.

\begin{lemma}\label{lem:upperbound-degree}
Let $\G$ be a group in $\GG$  and let $\phi\co \G\to \Z$ be an epimorphism. We write $\G'=\ker(\phi)$. Let $t\in\G$ be an element
with $\phi(t)=1$. 
Let $A$ and $B$ be matrices over $\Z[\G']$ with $k$ columns and $l$ rows.
Here $k\in \N$ and $l\in \N\cup \{\infty\}$. Suppose that all but $c$ rows of $B$ are zero. 
If 
\[ \dim_\phi\left( \Z[\G]^k/\Z[\G]^l(A+tB)\right)<\infty\]
then
\[ \dim_\phi\left( \Z[\G]^k/\Z[\G]^l(A+tB)\right)\leq c.\]
\end{lemma}

\begin{proof}
Let $\G$ be a group in $\GG$ and let $\phi\co \G\to \Z$ be an epimorphism. We write $\G'=\ker(\phi)$. Let $t\in\G$ be an element
with $\phi(t)=1$.
We write $K=\K(\G')$. We recall several facts and conventions established and collected  in \cite{Coc04,Ha05,Fr07}.
First of all, by \cite[Proposition~4.5]{Ha05} we can identify $\K(\G') \otimes_{\Z[\G']}\Z[\G]$ with a twisted Laurent polynomial ring 
$\kt$ over $\K:=\K(\G')$. 
For a $\Z[\G]$-module $M$ we have by definition that $\dim_\phi(M)=\dim_K(\kt \otimes_{\Z[\G]} M)$. 
We also note that by  \cite[Proposition~II.3.5]{Ste75} the ring $\kt$ is flat over $\Z[\G]$ since $\kt$ is a localization of $\Z[\G]$. 

Given $i$ in $\N$ we denote by $A_i$ and $T_i$ the $i\times k$-matrices which are given by the first $i$ rows of $A$ and $B$.
Now we have the following claim.

\begin{claim}
There exists an $i$ such that the projection map
\[ \kt^k/\kt^i(A_i+tT_i)\to  \kt^k/\kt^l(A+tB)\]
is an isomorphism.
\end{claim}

For $i\in \N\cup \{\infty\}$ we consider 
\[ S_i:=\mbox{the span over $\kt$ of the first $i$ rows of $A+tB$}.\]
(Since we view all modules as left modules we take of course the left-span of the first $i$ rows.) 
Note that $S_1,S_2,\dots $ is an ascending chain of $\kt$-left submodules of $\kt^k$. Since $\kt$ is a principal ideal domain we deduce from \cite[Proposition~1.21]{La91} that the chain $S_1,S_2,\dots $ stabilizes. Put differently, there exists  an $i$ with 
$S_i=S_{i+1}=\dots$. It thus follows that $S_i=\cup_{j}S_j=S_{\infty}$. 
This concludes the proof of the claim.

We note that for a finitely generated $\kt$-left module $V$ we have $\dim_K(V)<\infty$ if and only if $V$ is a $\kt$-torsion module.
It  follows from the flatness of $\kt$ and the above claim   that 
\[ \ba{rcl}  \dim_{\phi}\left(\Z[\G]^k/\Z[\G]^l(A+tB)\right)&=&\dim_{K}\left(\left(\kt\, \otimes_{\Z[\G]}\, \Z[\G]^k/\Z[\G]^l(A+tB)\right)\right)\\[2mm]
&=&\dim_{K}\left(\kt^k/\kt^l(A+tB)\right)\\
&=&\dim_{K}\left(\kt^k/\kt^i(A_i+tT_i)\right).\ea \]
By assumption $\dim_{K}\left(\kt^k/\kt^i(A_i+tT_i)\right)$ is finite.
By the above this implies that the $\kt$-module $\kt^k/\kt^i(A_i+tT_i)$ is $\kt$-torsion. It follows from 
 \cite[Proof of Proposition~9.1]{Ha05} that 
  $\dim_{K}\left(\kt^k/\kt^i(A_i+tT_i)\right)$   is bounded above by $c$.
\end{proof}

In the following lemma we calculate the  dimension for a module presented by a $1\times 1$-matrix.

\begin{lemma}\label{lem:upperbound-degree-2}
Let $\G$ be a group in $\GG$ and let $\phi\co \G\to \Z$ be an epimorphism. Then given any non-zero element $f$ of $\Z[\G]$ we have
\[ \dim_{\phi}(\Z[\G]/\Z[\G]f)=\th_\phi(\PP(f)).\]
\end{lemma}

\begin{proof}
Let $\G$ be a group in $\GG$ and let $\phi\co \G\to \Z$ be an epimorphism. We use some of the notation from the proof of Lemma~\ref{lem:upperbound-degree}. In particular we pick $t\in\G$ with $\phi(t)=1$ and we write $K=\K(\G')$. Furthermore we again  identify $\K(\G') \otimes_{\Z[\G']}\Z[\G]$ with a twisted Laurent polynomial ring 
$\kt$.

By sorting the summands of $f$ according to their $\phi$-values we can write  $f=\sum_{i=d}^D f_it^i$ where $f_i\in \Z[\G']$ with $f_d\ne 0$ and $f_D\ne 0$. It follows easily from the definitions that 
\[ \th_\phi(\PP(f))=D-d.\]
On the other hand we can use the usual argument from commutative Laurent polynomial rings with coefficients in a field to show that 
\[  \dim_{K}(\kt/\kt f) =D-d.\]
But as in the proof of Lemma~\ref{lem:upperbound-degree} we have $\dim_{K}(\kt/\kt f)=\dim_{\phi}(\Z[\G]/\Z[\G]f)$.
\end{proof}

\begin{lemma}\label{lem:euler}
Let $k\in \N$ and $l\in \N\cup \{\infty\}$. Let $\G$ be a group in $\GG$  and let $\phi\co \G\to \Z$ be an epimorphism. Let
\[ C_*:=\quad 0\to \Z[\G]^l\xrightarrow{\bp T_0& T_1\ep }\Z[\G]\oplus \Z[\G]^{k}\xrightarrow{\bp x_0\\ x_1\ep}\Z[\G]\to 0\]
be a chain complex. If $x_0\ne 0$ in $\Z[\G]$, then 
$\dim_\phi(H_0(C))$ is finite and 
\[ \dim_\phi(H_1(C_*))-\dim_\phi(H_0(C_*))=
\dim_\phi(\Z[\G]^k/\Z[\G]^l T_1)-\dim_\phi(\Z[\G]/\Z[\G]x_0).
\]
This equality means in particular that either both sides are finite or both are infinite. 
\end{lemma}

\begin{proof}
We again use the definitions and conventions from Lemma~\ref{lem:upperbound-degree}. 
By the flatness of $\kt$ we have $H_i(\kt \otimes_{\Z[\G]} C_*)=\kt \otimes_{\Z[\G]} H_i(C_*)$. We thus have to show that 
$H_0(C;\kt)$ is  a finite-dimensional $K$-vector space and that
\[ \ba{l}\hspace{1cm}\dim_K\left(H_1(C;\kt)\right)-\dim_K(H_0(C;\kt))\\[2mm]
\hspace{4cm}=\dim_K(\kt^k/\kt^lT_1)-\dim_K(\kt/\kt x_0).\ea\]
We consider the following commutative diagram:
\[ \xymatrix@C1.2cm@R1.5cm{\kt \otimes_{\Z[\G]} C_* =\hspace{-1.3cm} &\kt^l\ar[r]^-{\tiny \bp T_0& T_1\ep}\ar[d]&\kt\oplus \kt^{k}\ar[r]^-{\tiny \bp x_0\\ x_1\ep}\ar[d]&\kt\ar[d]\ar[r]& 0\\
\hspace{0.8cm}W_*=\hspace{-1.3cm} &0\ar[r]& \kt^k/\kt^lT_1\ar[r]^-{x_1}&\kt/\kt x_0\ar[r] & 0,}\]
where the vertical maps are given by the obvious projection maps.
It is straightforward to verify that the vertical maps induce isomorphisms between the homology groups in dimensions $1$ and $0$ of the chain complex
$\kt \otimes_{\Z[\G]} C_*$  on top and the chain complex $W_*$ on the bottom. Put differently, we have
\be \label{equ:samedim} H_i(\kt \otimes_{\Z[\G]} C_*)\cong H_i(W_*)\mbox{ for }i=0,1.\ee 

By assumption $x_0\ne 0$ in $\kt$, in particular $\kt/\kt x_0$ is a finite-dimensional $K$-vector space.
This implies immediately that 
\[ H_0(C;\kt)\cong \kt/\left((\kt\oplus \kt^k)\bp x_0\\ x_1\ep \right)\]
 is a finite-dimensional $K$-vector space.
This also implies that 
$$H_1(C;\kt)\cong \ker\Big(\cdot x_1\co \kt^k/\kt T_1\to \kt/\kt x_0\Big)$$ 
is finite-dimensional if and only if $\kt/\kt T_1$ is finite-dimensional.

Now we suppose that $\kt/\kt T_1$ is finite-dimensional. The bottom sequence of the above commutative diagram is thus a map between two finite-dimensional $K$-vector spaces. It thus follows that the difference between the dimensions of the kernel and the cokernel equals the difference between the dimensions of the vector spaces, i.e.\ we have
\[ \dim_K(H_1(W_*))-\dim_K(H_0(W_*))=\dim_K(\kt^k/\kt T_1)-\dim_K(\kt/\kt x_0).\]
The lemma follows from (\ref{equ:samedim}).
\end{proof} 

\begin{proof}[Proof of Proposition~\ref{prop:splittings-lower-bound}]
Let  $G$ be a group which admits a nice $(2,1)$--presentation. Furthermore,  let $\phi\co G\to \Z$ be an epimorphism. We write $G'=\ker(\phi)$. 

Let $\varphi\co G\to \G$ be an admissible homomorphism
such that $\tau(X_\pi,\varphi)\ne 0$. We denote the homomorphism 
\[ \G\to H_1(\G;\Z)/\mbox{torsion} \xleftarrow{\cong} G\xrightarrow{\phi}\Z\]
 again by $\phi$. Furthermore we write $\G'=\ker(\phi\co \G\to \Z)$. 

\begin{claim}
We have
\[\th_\phi(\PP(\tau(X_\pi,\varphi)))=\dim_\phi(H_1(X_\pi;\Z[\G]))-\dim_\phi(H_0(X_\pi;\Z[\G])).\]
\end{claim}

Now we consider the chain complex
$C_*^\varphi(X_\pi;\Z[\G])$ which with respect to the obvious bases is given by
\[ 0\to \Z[\G]\xrightarrow{\bp \varphi(r_x)&\varphi(r_y)\ep }\Z[\G]^2\xrightarrow{\bp \varphi(x-1)\\\varphi(y-1)\ep}\Z[\G]\to 0.\]
Recall that we assume that $\varphi(x)$ is non-trivial, i.e.\ $\varphi(x-1)$ is non-zero in $\Z[\G]$.  Thus we can apply  Lemma~\ref{lem:euler} to the chain complex
$C_*^\varphi(X_\pi;\Z[\G])$  and we obtain that
\[
\ba{l}
\hspace{1cm} \dim_\phi\big(H_1(X_\pi;\Z[\G])\big)
-\dim_\phi\big(H_0(X_\pi;\Z[\G])\big)\\[2mm]
\hspace{4cm} =\dim_\phi\big(\Z[\G]/\Z[\G]\varphi(r_y)\big)
- \dim_\phi\big(\Z[\G]/\Z[\G]\varphi(x-1)\big).\ea\]
But by Lemma~ \ref{lem:upperbound-degree-2} the latter difference equals precisely
\[ \th_\phi(\PP(\varphi(r_y)))-\th_\phi(\PP(\varphi(x-1)))=
\th_\phi(\PP(\tau(X_\pi,\varphi))). \]
This concludes the proof of the claim.

Now we write $c=c(G,\phi)$. Let
\[ f\co G\xrightarrow{\cong} \ll A,t\,|\, \mu(B)=tBt^{-1}\rr \]
be a splitting of $(G,\phi)$ over a finitely generated group $B$
with rank $c$. 
We pick a presentation $\ll g_1,\dots,g_k\,|\, r_1,r_2,\dots \rr$ for $A$ and we pick a finite generating set $x_1,\dots,x_c$ for $B$.  Note that
\[ \ba{ll} &\ll g_1,\dots,g_k,t\,|\, r_1,r_2,\dots,\mbox{ and }\mu(y)=tbt^{-1}\mbox{ for all }b\in B\rr\\
=&\ll g_1,\dots,g_k,t\,|\, r_1,r_2,\dots,\mu(x_1)^{-1}tx_1t^{-1},\dots,\mu(x_{c})^{-1}tx_{c}t^{-1}\rr.\ea\]
We denote by $l \in \N\cup \{\infty\}$ the number of relators in the second presentation.
We denote by $Y$ the 2-complex corresponding to this presentation of $G$. 
It has one 0-cell, $k+1$ 1-cells and $l$ 2-cells.
Also note that $\pi_1(Y)=\pi_1(X)$ and we thus have 
\[ \dim_\phi(H_i^\varphi(X_\pi;\Z[\G]))=\dim_\phi(H_i^\varphi(Y;\Z[\G]))\mbox{ for }i=0,1.\]
In light of the previous claim it thus suffices to prove the following claim.

\begin{claim}
We have
\[ c-1\geq \dim_\phi(H_1(Y;\Z[\G]))-\dim_\phi(H_0(Y;\Z[\G])).\]
\end{claim}

We denote by $M$ the matrix over $\Z[G]$ that is given by all the Fox derivatives of the relators. 
We denote the first column of $M$, corresponding to the Fox derivatives with respect to $t$, by $M_0$, and we denote the matrix given by all the other columns by $M_1$. 

We make the following observations.
\bn
\item The relators $r_1,r_2,\dots$ are words in $g_1,\dots,g_k$. The Fox derivatives of the $r_i$ with respect to the $g_j$ thus lie in $\Z[G']$.
\item For any $i\in \{1,\dots,k\}$ and $j\in \{1,\dots,{c}\}$ we have
\[ \frac{\partial}{\partial g_i}\left(\mu(x_j)^{-1}tx_jt^{-1}\right)=\frac{\partial}{\partial g_i}\left(\mu(x_j)^{-1}\right)+\mu(x_j)^{-1}t\frac{\partial}{\partial g_i}x_j.\]
The same argument as in (1) shows that the first term lies in $\Z[G']$, and one can similarly see that the second term is of the form $t\cdot g$, where $g\in \Z[G']$.
\en
Thus $M_1$ is of the form
\[ M_1=P_1+tQ_1,\]
where $P_1$ and $Q_1$ are matrices over $\Z[G']$, and where all but the last ${c}$ rows of $Q_1$ are zero.

By a slight abuse of notation we denote $\varphi(t)\in \G$ again by $t$.
Now we consider the chain complex $C_*^\varphi(Y;\Z[\G])$ with respect to the obvious bases:
\[ 
\Z[\G]^l\xrightarrow{\displaystyle \varphi(M_0)\oplus (\varphi(P)+t\varphi(Q))}\Z[\G]\oplus\Z[\G]^{k}\xrightarrow{\bp \varphi(t-1)\\\varphi(g_1-1)\\\dots\\\varphi(g_k-1)\ep}\Z[\G]\to 0.\]
Note that $t$ is non-trivial in $\G$ since $\phi$ factors through $\G$.
We can thus apply  Lemma~\ref{lem:euler} and we obtain that 
\[ \ba{l}
\hspace{1cm}
\dim_\phi\big(H_1(Y;\Z[\G])\big)-\dim_\phi\big(H_0(Y;\Z[\G])\big)\\
\hspace{4cm}=\dim_\phi\big(\Z[\G]^k/\Z[\G]^l(\varphi(P)+t\varphi(Q))\big)-\dim_\phi\big(\Z[\G]/\Z[\G]\varphi(t-1)\big).\ea\]
By  Lemma~\ref{lem:upperbound-degree} we have 
\[ c\geq\dim_\phi\left(\Z[\G]^k/\Z[\G]^l(\varphi(P)+t\varphi(Q))\right) \]
and by Lemma~\ref{lem:upperbound-degree-2} we have $\dim_\phi(\Z[\G]/\Z[\G]\varphi(t-1))=1$.
This concludes the proof of the claim and thus of the proposition.
\end{proof}

\section{Groups which admit a $(2,1)$--presentation with $b_1=1$}\label{section:b1}
Throughout the paper we worked with with nice $(2,1)$--presentations, i.e.\ with presentations $\pi=\ll x,y|r\rr$  where $r$ is non-empty and cyclically reduced and with $b_1(G_\pi)=2$.

Now we will see that we can drop the condition $b_1(G_\pi)=2$ in almost all cases.
Before we state the next proposition we need to introduce two more definitions.
\bn
\item For $m,n\in \Z$ the Baumslag-Solitar group $B(m,n)$ is defined as 
\[ B(m,n):=\ll x,y| xy^mx^{-1}=y^n\rr.\]
\item We say a $(2,1)$--presentation $\pi=\ll x,y\,|\,r \rr$ is \emph{simple} if  $b_1(G_\pi)=1$, if  $x$ defines a generator of $H_1(\pi;\Z)/\mbox{torsion}$ and if $y$ represents the trivial element in $H_1(\pi;\Z)/\mbox{torsion}$.
\en
Now we can formulate the following proposition.

\begin{proposition}\label{prop:foxpolytope-b1}
Let $\pi=\ll x,y\,|\,r \rr$
be a  $(2,1)$--presentation  where $r$ is non-trivial and cyclically reduced.
If $\pi$ is simple and if   $G_\pi$ is not isomorphic to $B(\pm 1,n)$ for any $n\in \Z$, then  there exists a unique marked polytope $\MM$, such that 
\[ \MM+\MM(x-1)= \MM(r_y).\]
\end{proposition}

\begin{proof}
By our hypothesis there exists  an epimorphism $\phi\co \pi\to \Z$ with $\phi(x)=1$ and $\phi(y)=0$.  We use this epimorphism to identify $H_1(\pi;\R)$ with $\R$.
 
Note that $\MM(x-1)$ is an interval of length $1$ where both end points are marked. 
Also note that  $r$ is either of the form
$x^{m_1}y^{n_1}\cdot \dots \cdot x^{m_k}y^{n_k}$ where $m_1,n_1,\dots,m_k,n_k$ are non-zero or it is of the form
$y^{n_1}x^{m_1}\cdot \dots \cdot y^{n_k}x^{m_k}$ where $m_1,n_1,\dots,m_k,n_k$ are non-zero.  In either case our assumptions on $x$ and $y$ imply that $m_1+\dots+m_k=0$.

Given $i\in \{1,\dots,k\}$ we write 
\[ \ba{rcl} M_i&=&m_1+\dots+m_i, \\
D&=&\max\{M_1,\dots,M_k\}-\min\{M_1,\dots,M_k\}.\ea \]
It follows easily from the definitions that  $\PP(r_y)$ is an interval of length $D$.

If $D\geq 2$, then we denote by $\MM$ a marked interval in $\R$ with length $D-1$ where we mark the left (respectively right) vertex if and only if the left (respectively right) vertex of $\MM(r_y)$ is marked. After possibly translating $\MM$ by an integer we then have $\MM+\MM(x-1)=\MM(r_y)$. 

Now we consider the case that $D=1$. It follows that the $m_i$ are alternating between $1$ and $-1$.  Since $m_1+\dots+m_k=0$ we deduce that $k$ is even.
It follows easily from Corollary~\ref{cor:distinct}
that both end points of $\MM(r_y)$ are not marked unless $k=2$ and at least one of $n_1$ or $n_2$ is equal to $\pm 1$. But this case does not occur, since such a group would be isomorphic to a Baumslag-Solitar group of the form $B(\pm 1,n)$. Summarizing, we showed that $\MM(r_y)$ is an interval of length one such that both end points are not marked. In this case we  take  $\MM$ to be the polytope which consists of a single not marked point. It is clear that this $\MM$ has the desired property and that it is unique up to translation.
\end{proof}

For a $(2,1)$--presentation as in Proposition~\ref{prop:foxpolytope-b1} we now define $\MM_\pi$ to be the marked polytope that we found in that proposition.

Finally let  $\pi=\ll x,y\,|\,r \rr$ be any  $(2,1)$--presentation  where $r$ is non-trivial and cyclically reduced and with $b_1(G_\pi)=1$.
We can apply the proof of Lemma~\ref{lem:splittings1} verbatim to $\pi$ and we obtain a simple presentation $\pi'=\ll x',y'\,|\,r'\rr,$  where $r'$ is non-trivial and cyclically reduced. If $G_\pi\cong G_{\pi'}$ is not isomorphic to $B(\pm 1,n)$, then we define $\MM_\pi:=\MM_{\pi'}$. 

Now it is straightforward to verify that the statements of Theorems \ref{mainthm}, \ref{mainthm2} and \ref{mainthm3} also hold in this context.
We leave the details to the reader. 

Finally, note that it is not possible to find a marked polytope for the Baumslag-Solitar groups $B(\pm 1,n), n\ne \pm 1$ which satisfies the conclusions of Theorems \ref{mainthm} and \ref{mainthm3}. Indeed, for Theorem~\ref{mainthm3} to hold the polytope would have to consist of a single point. But the  Bieri--Neumann--Strebel invariant contains one epimorphism $\phi\co \pi\to \Z$ but not the other. So the one vertex of the polytope would have to be marked and not marked at the same time.

\section{Conclusion and questions}\label{section:questions}
Given a group $G$ with  a nice $(2,1)$--presentation $\pi$
we  used Fox calculus to define a marked polytope $\MM_\pi$ that in particular determines  the Bieri--Neumann--Strebel invariant of $G_\pi$. We also showed that in many cases $\MM_\pi$ carries interesting further information on $G$. 
 It remains an open problem to relate the polyhedral structure of the polytope $\PP_\pi$ to properties of the group $G_\pi,$ and to extend the construction to more general classes of groups. We conclude this paper with several questions aimed at this.

\begin{question}
Is the polytope $\PP_\pi$ an invariant of the underlying group $G_\pi$?
\end{question}

In Theorem~\ref{mainthm3} we proved that if $G$ is residually a torsion-free elementary amenable group, then the thickness of $\PP_\pi$ for any epimorphism $\phi\co G\to \Z$ can be described purely in terms of $G_\pi$ and $\phi$.
This does not give an intrinsic definition of the polytope since by Lemma~\ref{lem:addth} (2) the thickness only determines the \emph{symmetrization} of  $\PP_\pi$.

\begin{question}
Is there an intrinsic definition of the  polytope $\PP_\pi$?
\end{question}

The Bieri--Neumann--Strebel-invariant $\Sigma(G_\pi)$ can be identified with an open subset contained in the interior of the faces of $\PP_\pi.$ Moreover, the set of all points in $\Sigma(G_\pi)$ corresponding to homomorphisms with finitely generated kernel is symmetric and open. In this way, given an asymmetric marked polytope, the vertices determine a natural subdivision of some of the opposite faces, and hence a possibly finer polyhedral structure, with some open regions corresponding to finite generation.

\begin{question}
Does the polyhedral structure of the polytope $\PP_\pi$  contain more information about $G_\pi$? 
\end{question}

\begin{question}
Is it possible to assign to any finitely presented group
a marked polytope which satisfies the conclusions of Theorems  \ref{mainthm}
and \ref{mainthm3}?
More modestly, one could ask for a marked polytope for groups with a presentation of deficiency one.
\end{question}

If $G$ is a group with a 2--dimensional Eilenberg-Maclane space of zero Euler characteristic, then  the approach of Section~\ref{section:tg} 
together with a variation on Proposition~\ref{prop:invtofg}
will assign to $G$ a  (possibly empty) polytope.
However, if $G$ is not residually $\GG,$ then it is unlikely that the polytope will have the desired properties.

Over the last years a lot of effort has been put into understanding free-by-cyclic groups, i.e.\  groups of the form $\Z\ltimes_\varphi F,$ where $F$ is a free group and $\varphi\co F\to F$ is an isomorphism. These groups have a presentation of deficiency one, and they
are residually $\GG$ by Lemma~\ref{lem:zbyfree}. 
The construction of Section~\ref{section:tg}  will then actually give a non-empty polytope and it should be interesting to relate it to aspects of \cite{DKL13a,DKL13b} and \cite{AKHR13}. For example, if $\phi\in H^1(G_\pi;\R)$ has the property that both $\phi$ and $-\phi$ lie in $\Sigma(G_\pi)$, then both approaches `see' the function $\psi\mapsto 1-\rank(\ker(\psi))$ in a neighborhood of $\phi$. 

Finally, let $\pi$ be a presentation such that each generator appears at least twice in the relators.
(Note that a nice $(2,1)$--presentation is of that type.)
Given such a presentation $\pi$  Turaev  \cite{Tu02} defined a seminorm on $H^1(X_\pi;\R)$. We conclude this paper with the following question.

\begin{question}
 Let $G_\pi$ be a nice $(2,1)$--presentation. Is the polytope $\PP_\pi$ in $H_1(G_\pi;\R)=H_1(X_\pi;\R)$ dual to the unit norm ball of the norm defined by Turaev \cite{Tu02} on $H^1(X_\pi;\R)=\hom(H_1(X_\pi;\R),\R)$?
\end{question}


\begin{thebibliography}{10}
\bibitem[AKHR13]{AKHR13}
Y. Algom--Kfir, E. Hironaka and K. Rafi, {\em
Digraphs and cycle polynomials for free-by-cyclic groups}, Preprint (2013)
\bibitem[AFW13]{AFW13}
M. Aschenbrenner, S. Friedl and H. Wilton, {\em 3--manifold groups}, Preprint (2013)
\bibitem[Bi07]{Bi07}
R. Bieri, {\em Deficiency and the geometric invariants of a group}, with an appendix by Pascal
Schweitzer, J. Pure Appl. Algebra 208 (2007), 951--959.
\bibitem[BNS87]{BNS87}
R. Bieri, W. D. Neumann and R. Strebel, {\em A geometric invariant of discrete groups}, Invent.
Math. 90 (1987), 451-477.
\bibitem[BR88]{BR88}
R. Bieri and B. Renz, {\em Valuations on free resolutions and higher geometric invariants of groups},
Comment. Math. Helv. 63 (1988), 464--497.
\bibitem[BS78]{BS78} R. Bieri and R. Strebel, {\em Almost finitely presented soluble groups}, Comment. Math. Helv. 53 (1978), 258--278. 
\bibitem[Brn87]{Brn87}
K. S. Brown, {\em Trees, valuations, and the Bieri-Neumann-Strebel invariant}, Invent. Math. 90
(1987), 479-504.
\bibitem[Brj80]{Brj80}
S. D. Brodskij, {\em Equations  over  groups  and  groups  with  a  single  defining  relation},
Uspekhi  Mat.  Nauk  354 (1980), 183, Russian  Math. 
Surveys  354 (1980), 165.
\bibitem[Brj84]{Brj84}
S. D. Brodskij,  {\em   Equations  over  groups,  and  groups  with  one  defining  relation},  Sib. 
Mat.  Zh.  25:2,  8400103  (1984).   Sib.  Math.  J.  25,  235--251  (1984). 
\bibitem[Coc04]{Coc04}
T. Cochran,  {\em Noncommutative knot theory}, Algebr. Geom. Topol. 4 (2004), 347--398.
\bibitem[Coh85]{Coh85}
P. M. Cohn, {\em Free Rings and their Relations}, Second Edition, London Math. Soc.
Monographs no. 19, Academic Press, London and New York 1985.
\bibitem[CZ93]{CZ93}
D. J. Collins and H. Zieschang, {\em Combinatorial group theory and fundamental groups},
Algebra, VII, 1--166, 233-240, Encyclopaedia Math. Sci., 58, Springer, Berlin, 1993.
\bibitem[DLMSY03]{DLMSY03}
J. Dodziuk, P. Linnell, V. Mathai, T. Schick and S. Yates, {\em Approximating $L^2$-invariants,
and the Atiyah conjecture}, Preprint Series SFB 478 M\"unster, Germany. Communications on Pure
and Applied Mathematics, vol. 56, no. 7 (2003), 839-873.
\bibitem[DKL13a]{DKL13a}
S. Dowdall, I. Kapovich and C. J. Leininger, {\em 
Dynamics on free-by-cyclic groups}, Preprint (2014)
\bibitem[DKL13b]{DKL13b}
S. Dowdall, I. Kapovich and C. J. Leininger, {\em 
McMullen polynomials and Lipschitz flows for free-by-cyclic groups}, Preprint (2014)
\bibitem[Du01]{Du01}
N. Dunfield, {\em Alexander and Thurston norms of fibered
3--manifolds}, Pacific J. Math. {200} (2001), no. 1,
43--58.
\bibitem[DT06]{DT06}
N. Dunfield and D. Thurston, {\em
A random tunnel number one $3$-manifold does not fiber over the circle},
Geom. Topol. 10 (2006), 2431--2499.
\bibitem[FGS10]{FGS10}
M. Farber, R. Geoghegan and D. Sch\"utz, {\em
Closed 1-forms in topology and geometric group theory},
Russian Mathematical Surveys, 65:1 (2010), 145--176.
\bibitem[Fo53]{Fo53}
R. H. Fox, {\em Free differential calculus I, Derivation in the free group ring}, Ann. Math. 57 (1953), 547--560.
\bibitem[Fr07]{Fr07}
S. Friedl, {\em Reidemeister torsion, the Thurston norm and Harvey's invariants},
Pac.  J. Math.  230 (2007), 271--296.
\bibitem[FH07]{FH07}
S. Friedl and S. Harvey, {\em Non--commutative multivariable Reidemeister torsion and the Thurston
norm}, Alg. Geom. Top.  7 (2007), 755--777.
\bibitem[FSW13]{FSW13}
S. Friedl, D. Silver and S. Williams, {\em Splittings of knot groups}, Preprint (2013), to be published by Math. Ann.
\bibitem[FT15]{FT15}
S. Friedl and S. Tillmann, {\em 3--manifold groups with two generators and one relation},
in preparation (2015)
 \bibitem[FV10]{FV10}
S. Friedl and  S. Vidussi,
 {\em A survey of twisted Alexander polynomials}, The Mathematics of Knots: Theory and Application (Contributions in Mathematical and Computational Sciences), editors: Markus Banagl and Denis Vogel (2010),  45--94.
\bibitem[Ha05]{Ha05}
S. Harvey, {\em Higher--order polynomial invariants of 3--manifolds giving lower bounds for the
Thurston norm}, Topology 44 (2005), 895--945.
\bibitem[Hi40]{Hi40}
G. Higman, {\em The units of group-rings}, Proc. London Math. Soc. (2) 46 (1940), 231--248.
\bibitem[Ho00]{Ho00}
J. Howie, {\em A short proof of a theorem of Brodskii}, Publ. Mat. Univ. Aut. Barcelona 44 (2000), 641--647.
\bibitem[KLM88]{KLM88}
P. H. Kropholler, P. A. Linnell and J. A. Moody, {\em Applications of a new $K$-theoretic theorem to
soluble group rings}, Proc. Amer. Math. Soc.  104 (1988),  no. 3, 675--684.
\bibitem[La91]{La91}
T. Y. Lam, {\em A First Course in Noncommutative Rings}, Graduate Texts in Mathematics, Springer Verlag (1991)
\bibitem[LM06]{LM06}
C. Leidy and L. Maxim, {\em Higher-order Alexander invariants of plane algebraic curves},
 Int. Math. Res. Not.  2006 (2006), Article ID 12976, 23 pages.
\bibitem[LM08]{LM08}
C. Leidy and L. Maxim, {\em
Obstructions on fundamental groups of plane curve complements},
Real and Complex Singularities, Contemporary Mathematics 459 (2008), 117--130.
\bibitem[LS77]{LS77}
R. Lyndon and P. Schupp, {\em Combinatorial group theory}, Ergebnisse der Mathematik und ihrer Grenzgebiete, Band 89. Springer-Verlag, Berlin-New York, 1977.
\bibitem[MP73]{MP73}
J. McCool and A. Pietrowski, {\em On a conjecture of W. Magnus},
Word Probl., Decision Probl. Burnside Probl. Group Theory, Studies Logic Foundations Math. 71 (1973), 453--456.
\bibitem[No81]{No81}
S. P. Novikov, {\em Multi-valued functions and functionals. An analogue of Morse theory},
Soviet Math. Doklady 24 (1981), 222--226.
\bibitem[Pa77]{Pa77}
D. S. Passman, {\em The algebraic structure of group rings},
John Wiley \& Sons. XIV (1977)
\bibitem[Ro94]{Ro94} J. Rosenberg, {\em Algebraic K -theory and its applications}, Graduate Texts in Mathematics 147, Springer, New York, 1994. 
\bibitem[Sc93]{Sc93}
R. Schneider, {\em Convex bodies: the Brunn-Minkowski theory}, Cambridge Univ. Press (1993)
\bibitem[Si87]{Si87}
J.--C. Sikorav, {\em Homologie de Novikov associ\'ee \`a une classe de cohomologie r\'eelle de degr\'e un}, Th\`ese Orsay, 1987.
\bibitem[Ste75]{Ste75}
B. Stenstr\"om, {\em Rings of Quotients}, Springer-Verlag, 1975.
\bibitem[Str84]{Str84} R. Strebel, {\em Finitely presented soluble groups}, in Group Theory: Essays for Philip Hall, Academic Press, London, 1984. 
\bibitem[Th86]{Th86}
W. P. Thurston, {\em A norm for the homology of 3--manifolds}, Mem. Amer. Math. Soc. 59, no.
339 (1986), 99--130.
\bibitem[Tu01]{Tu01}
V. Turaev, {\em Introduction to Combinatorial Torsions}, Lectures in Mathematics, ETH Z\"urich
(2001)
\bibitem[Tu02]{Tu02}
V. Turaev, {\em A norm for the cohomology of 2-complexes}, Alg. Geom. Top. 2 (2002) 137--155.
\bibitem[Wa78]{Wa78}
F. Waldhausen, {\em  Algebraic K-theory of generalized free products II}, Ann. of
Math. 108 (1978), 135--256.
\bibitem[We72]{We72}
C. Weinbaum, {\em On relators and diagrams for groups with one defining relation},
Illinois J. Math. 16 (1972), 308--322.
\bibitem[Zi70]{Zi70}
H. Zieschang, {\em \"Uber die Nielsensche K\"urzungsmethode in freien Produkten mit Amalgam},
Invent. Math. 10 (1970), 4--37.
\end{thebibliography}
\end{document}